\documentclass[10pt]{amsart}
\usepackage{amssymb}
\usepackage{mathrsfs}
\usepackage{amsfonts}
\usepackage{amssymb,amsfonts}

\allowdisplaybreaks

\begin{document}
\theoremstyle{plain}
\newtheorem{thm}{Theorem}[section]
\newtheorem{theorem}[thm]{Theorem}
\newtheorem{lemma}[thm]{Lemma}
\newtheorem{corollary}[thm]{Corollary}
\newtheorem{proposition}[thm]{Proposition}
\newtheorem{conjecture}[thm]{Conjecture}
\newtheorem{obs}[thm]{}
\theoremstyle{definition}
\newtheorem{construction}[thm]{Construction}
\newtheorem{notations}[thm]{Notations}
\newtheorem{question}[thm]{Question}
\newtheorem{problem}[thm]{Problem}
\newtheorem{remark}[thm]{Remark}
\newtheorem{remarks}[thm]{Remarks}
\newtheorem{definition}[thm]{Definition}
\newtheorem{claim}[thm]{Claim}
\newtheorem{assumption}[thm]{Assumption}
\newtheorem{assumptions}[thm]{Assumptions}
\newtheorem{properties}[thm]{Properties}
\newtheorem{example}[thm]{Example}
\newtheorem{comments}[thm]{Comments}
\newtheorem{blank}[thm]{}
\newtheorem{defn-thm}[thm]{Definition-Theorem}

\newcommand{\sM}{{\mathcal M}}
\newcommand{\cK}{{\mathcal K}}
\def\res{\operatorname{Res}}
\renewcommand{\arraystretch}{1.3}


\title[$\Psi$DO and Witten's $r$-spin numbers]{Formal pseudodifferential operators and \\Witten's $r$-spin numbers}

\author{Kefeng Liu}
        \address{Center of Mathematical Sciences, Zhejiang University, Hangzhou, Zhejiang 310027, China;
                Department of Mathematics,University of California at Los Angeles,
                Los Angeles}
        \email{liu@math.ucla.edu, liu@cms.zju.edu.cn}

        \author{Ravi Vakil}
        \address{Department of Mathematics, Stanford University}
        \email{vakil@math.stanford.edu}

        \author{Hao Xu}
        \address{Department of Mathematics, Harvard University}
        \email{haoxu@math.harvard.edu}

        \begin{abstract}We derive an effective recursion for Witten's
          $r$-spin intersection numbers, using Witten's
          conjecture relating $r$-spin numbers to the Gel'fand-Dikii
          hierarchy (Theorem~\ref{spin}).
 Consequences include
          closed-form descriptions of the intersection numbers
(for example, in
         terms of gamma functions:  Propositions~\ref{onept2}
          and~\ref{onept}, Corollary~\ref{c}).
We use these
         closed-form descriptions
 to prove Harer-Zagier's
formula for the Euler characteristic of $\mathcal M_{g,1}$.
Finally in
\S \ref{sectionzero}, we extend Witten's series expansion formula for the Landau-Ginzburg
potential to study $r$-spin numbers in the small phase space in
genus zero.
Our key tool is the calculus of formal pseudodifferential
          operators, and is partially motivated by work of Br\'ezin
          and Hikami.
        \end{abstract}


    \maketitle
\tableofcontents

\section{Introduction}

\label{s:introduction}Motivated by two dimensional gravity, E. Witten proposed two
influential conjectures relating integrable hierarchies to the
intersection theory of moduli spaces of curves, see \cite{Wi, Wi2}.

We begin by recalling Witten's definition of $r$-spin intersection numbers.
Witten's original papers \cite{Wi2, Wi3} remain the best
introduction to the mathematical and physical background of this
subject. Other excellent expositions can be found in \cite{JKV, Sh}. For
an introduction to relevant facts about the moduli spaces of curves, see \cite{Va}.

Let $\Sigma$ be a Riemann surface of genus $g$ with marked points
$x_1$, $x_2$, \dots, $x_s$. Fix an integer $r\geq 2$. Label each marked
point $x_i$ by an integer $m_i$, $0\leq m_i\leq r-1$. Consider the
line bundle $\mathcal S=\cK \otimes \mathcal
O(- \sum_{i=1}^s m_i x_i)$ over $\Sigma$, where $\cK$ as usual denotes the canonical line
bundle. If $2g-2-\sum_{i=1}^s m_i$ is divisible by $r$, then there
are $r^{2g}$ isomorphism classes of line bundles $\mathcal T$ such
that $\mathcal T^{\otimes r}\cong \mathcal S$. The choice of an
isomorphism class of $\mathcal T$ determines a finite \'etale cover $\mathcal
M^{1/r}_{g,s}$ of $\mathcal M_{g,s}$, the moduli space of $r$-spin
curves, which comes with a universal curve $\pi: \mathcal{C}^{1/r}_{g,n} \rightarrow
\mathcal{M}^{1/r}_{g,s}$, on which lives a universal bundle, which we also
sloppily denote $T$. A compactification of
$\mathcal M^{1/r}_{g,s}$, denoted by $\overline{\mathcal
M}^{1/r}_{g,s}$, was constructed in \cite{AJ, Ja}.

Let $\mathcal V$ be a vector bundle over $\overline{\mathcal
M}^{1/r}_{g,s}$ whose fiber is the dual space to
$H^1(\Sigma,\mathcal T)$.  More precisely, $\mathcal V := R^1 \pi_*
\mathcal T$. The top Chern class $c_{top}(\mathcal{V})$
of this bundle has degree $(g-1)(r-2)/r+\sum_{i=1}^sm_i/r$. The
algebro-geometric constructions of $c_{top}(\mathcal{V})$ can be
found in \cite{Chi2, PV}.

We associate with each marked point $x_i$ an integer $n_i\geq 0$.
Witten's $r$-spin intersection numbers are defined by
\begin{equation}\label{eqrspin}
\langle\tau_{n_1,m_1}\dots\tau_{n_s,m_s}\rangle_g=
\frac{1}{r^g}\int_{\overline{\mathcal
M}^{1/r}_{g,s}}\prod_{i=1}^s\psi(x_i)^{n_i}\cdot
 c_{top}(\mathcal{V}),
\end{equation}
which is non-zero only if
\begin{equation} \label{eqdim}
(r+1)(2g-2)+rs=r\sum_{j=1}^s n_j+\sum_{j=1}^s m_j.
\end{equation}

Fix an integer $r\geq2$. Consider the pseudodifferential operator
\begin{equation} \label{eqpdo7}
Q=D^r+\sum_{i=0}^{r-2}\gamma_i(x)D^i,  \quad \text{ where }
D=\frac{\sqrt{-1}}{\sqrt{r}}\frac{\partial}{\partial x}.
\end{equation}

It is easy to see that there is a unique pseudodifferential operator
$L$ such that $L^r=Q$ (see Lemma \ref{pdo3}),which we denote
$$
Q^{1/r}=D+\sum_{i>0}w_{-i}D^{-i},
$$ where the coefficients $\{w_{-i}\}$ are universal differential
polynomials in the $\{\gamma_i\}$.

The Gel'fand--Dikii equations read
$$
i\frac{\partial Q}{\partial
t_{n,m}} = [Q^{n+(m+1)/r}_+,Q]\cdot \frac{c_{n,m}}{\sqrt{r}},
$$
where the constants $c_{n,m}$ are given by
$$
c_{n,m}=\frac{(-1)^{n}r^{n+1}} {(m+1)(r+m+1)\cdots(nr+m+1)}.
$$

Consider the formal series $F$ in variables $t_{n,m}$, $n\geq 0$ and
$0\leq m\leq r-1$,
$$
F(t_{0,0},t_{0,1},\dots)=\sum_{d_{n,m}}\langle\prod_{n,m}\tau_{n,m}^{d_{n,m}}\rangle
\prod_{n,m}\frac{t_{n,m}^{d_{n,m}}}{d_{n,m}!}.
$$

Witten conjectured in \cite{Wi2} that the above $F$ is the
string solution of the $r$-Gel'fand--Dikii hierarchy, namely  that $F$
satisfies
\begin{equation} \label{eqwitten}
\frac{\partial ^2 F}{\partial t_{0,0} \partial t_{n,m}}
=-c_{n,m}\res (Q^{n+\frac{m+1}{r}}),
\end{equation}
where $Q$ satisfies the Gel'fand--Dikii equations and $t_{0,0}$ is
identified with $x$. In addition, $F$ satisfies the string equation
\begin{equation}\label{eqstr}
\frac{\partial F}{\partial t_{0,0}}=\frac{1}{2}
\sum_{i,j=0}^{r-2}\delta_{i+j,r-2}t_{0,i}t_{0,j}+
\sum_{n=0}^{\infty} \sum_{m=0}^{r-2} t_{n+1,m} \frac{\partial
F}{\partial t_{n,m}}.
\end{equation}
This should be regarded as a boundary condition for $F$.

When $r=2$, the above assertion is the celebrated Witten-Kontsevich
theorem \cite{Ko}, to which there are a number of enlightening proofs. Witten's conjecture for any $r\geq 2$ has been
proved by Faber, Shadrin and Zvonkine \cite{FSZ}, building on
work of Givental and Lee \cite{Lee}. In fact, Witten's $r$-spin theory
corresponds to $A_{r-1}$ singularity in the Landau-Ginzburg theory.
Fan, Javis and Ruan \cite{FJR} have developed a Gromov-Witten type
quantum theory for all non-degenerate quasi-homogeneous singularities
and proved the ADE-integrable hierarchy conjecture of Witten. Chang
and Li \cite{CL} have initiated a program to give an algebro-geometric
construction of Landau-Ginzburg theory.

Witten's  constraints \eqref{eqwitten} (the $r$-Gel'fand--Dikii
equation) and  \eqref{eqstr} (the string equation) uniquely
determine $F$. There is much interest in understanding the
structure of $r$-spin intersection numbers both in mathematics and
physics (cf. \cite{BH, BH2, Chi, KL, Na, SZ}).

The paper is organized as follows.   In \S \ref{sectionwitten},
we recall useful identities of $r$-spin numbers. In \S
\ref{sectionpdo}, we prove a structure theorem of formal
pseudodifferential operators and use it to derive/define ``universal differential
polynomials'' $W_r(z)$, which will play a central role in the rest of
the paper.    In \S \ref{sectionalgorithm}, we present
a recursive algorithm for computing Witten's $r$-spin numbers for
all genera.
 Consequences include
         closed-form descriptions of
the one-point $r$-spin numbers, which we use in
 \S \ref{sectioneuler} to prove Harer-Zagier's
formula for the Euler characteristic of $\mathcal M_{g,1}$. In
\S \ref{sectionzero}, we study $r$-spin numbers on small phase
spaces in genus zero.


\noindent{\bf Acknowledgements.}  We thank J. Li, W. Luo, M. Mulase,
Y.B. Ruan, and  J. Zhou  for helpful conversations. The third author
thanks Professor D. Zeilberger for answering a question on
computer proof of combinatorial identities.

\vskip 30pt
\section{Review:  Witten's $r$-spin intersection
numbers}\label{sectionwitten}
In this section, we collect fundamental properties of $r$-spin
intersection numbers that we will use in this paper. The proof of
the these identities can be found in \cite{Wi2,JKV}. The $r$-spin
numbers satisfy the following:

\begin{enumerate}
\item[i)] If $m_i=r-1$, for some $1\leq i \leq s$, then
$$
\langle \tau_{n_1,m_1} \cdots \tau_{n_s, m_s} \rangle_g = 0.
$$
\item[ii)] {\em (string equation)}
\begin{equation} \label{eqstring}
\langle \tau_{0,0}\prod_{i=1}^s\tau_{n_i, m_i}\rangle_g =
\sum_{j=1}^s \langle \tau_{n_{j}-1,m_j} \prod^s_{\substack{i=1\\i\ne
j}}\tau_{n_i, m_i} \rangle_g.
\end{equation}
This, along  with $\langle \tau_{0,0}\tau_{0, i} \tau_{0, j}
\rangle_0=\delta_{i+j,r-2}$, is equivalent to \eqref{eqstr}.
\item[iii)] {\em (dilaton equation)}
\begin{equation} \label{eqdilaton}
\langle \tau_{1,0} \prod_{i=1}^s \tau_{n_i, m_i}\rangle_g =
(2g-2+s)\langle \prod_{i=1}^s \tau_{n_i, m_s}\rangle_g.
\end{equation}
\item[iv)] {\em (genus zero topological recursion relation)}
\begin{multline} \label{eqzero}
\langle \tau_{n_1+1, m_1} \tau_{n_2, m_2} \tau_{n_3, m_3}
\prod_{i=4}^s \tau_{n_i, m_i}\rangle_0 = \sum_{\{4 \cdots
s\}=I\coprod J} \sum_{m', m''=0}^{r-2}\langle \tau_{n_1, m_1}
\prod_{i\in I} \tau_{n_i, m_i} \tau_{0, m'} \rangle_0 \\
\cdot\eta^{m', m''}\langle \tau_{0, m''} \tau_{n_2, m_2} \tau_{n_3,
m_3} \prod_{i\in J} \tau_{n_i, m_i}\rangle_0,
\end{multline}
 where $\eta^{m', m''}=\delta_{m'+m'', r-2}$.
\item[v)] {\em (WDVV equation in genus zero)}
\begin{multline}\label{eqwdvv}
\sum_{m',m''=0}^{r-2} \prod_{\{5\cdots s\}=I \coprod J}
\langle\tau_{n_1, m_1}\tau_{n_2, m_2} \prod_{i\in I}\tau_{n_i,
m_i}\tau_{0, m'}\rangle_0 \eta^{m', m''}\langle\tau_{0,
m''}\tau_{n_3, m_3} \tau_{n_4, m_4}\prod_{i\in J}\tau_{n_i,
m_i}\rangle_0 \\ = \sum_{m',m''=0}^{r-2} \prod_{\{5\cdots s\}=I
\coprod J} \langle\tau_{n_1, m_1}\tau_{n_3, m_3} \prod_{i\in
I}\tau_{n_i, m_i}\tau_{0, m'}\rangle_0 \eta^{m', m''}\langle\tau_{0,
m''}\tau_{n_2, m_2} \tau_{n_4, m_4}\prod_{i\in J}\tau_{n_i,
m_i}\rangle_0
\end{multline}
\end{enumerate}

Witten  gives a detailed study of $r$-spin numbers in
genus zero in \cite{Wi2}. As he points out, the genus zero topological
recursion relation can be used to eliminate all descendent indices
(those $\tau_{i,j}$ with $i>0$),
so we only need to consider primary intersection numbers $\langle
\tau_{0,m_1},\cdots \tau_{0, m_s} \rangle$ on the small phase space.
Witten proves that the WDVV equation uniquely determines
primary $r$-spin intersection numbers in genus zero. For the reader's
convenience, we record Witten's work below in a more explicit form.
We will denote $\langle \tau_{0,a_1},\cdots,\tau_{0,a_s}\rangle_0$
by either $\langle \tau_{a_1},\cdots,\tau_{a_s}\rangle$ or $\langle
a_1,\cdots,a_s\rangle$. Witten  proves that
\begin{align*}
\langle
\tau_{a_1}\tau_{a_2}\tau_{a_3}\rangle&=\delta_{a_1+a_2+a_3,r-2},\\
\langle
\tau_{a_1}\tau_{a_2}\tau_{a_3}\tau_{a_4}\rangle&=\frac{1}{r}\cdot
\min(a_i,r-1-a_i).
\end{align*}

\begin{theorem}[Witten, \cite{Wi2}]
Let $s\geq5$, $a_1\geq\cdots\geq a_s$ and $\sum_{j=1}^s
a_j=r(s-2)-2$. Define $z=a_1, y=a_2, x=a_3$ and
$$m_1=x+z-(r-1),\quad m_2=r-1-z,\quad m_3=y,\quad m_4=z.$$
Then Witten's formula can be written as
\begin{multline} \label{eqwitten2}
\langle a_1,\cdots,a_s\rangle=\bigg\langle
x+y+z-(r-1),r-1-z,z,\prod_{i=4}^s
a_i\bigg\rangle\\+\sum_{\substack{I\coprod
J=\{4,\dots,s\}\\I,J\ne\emptyset}}\sum_{j=0}^{r-2}\left(\bigg\langle
j,m_1,m_3,\prod_{i\in I}a_i\bigg\rangle\bigg\langle
r-2-j,m_2,m_4,\prod_{i\in J}a_i\bigg\rangle\right.\\\left.
-\bigg\langle j,m_1,m_2,\prod_{i\in I}a_i\bigg\rangle\bigg\langle
r-2-j,m_3,m_4,\prod_{i\in J}a_i\bigg\rangle\right).
\end{multline}
This formula recursively computes all primary $r$-spin numbers.
\end{theorem}
\begin{proof} The argument is due to Witten.
From $s\geq 5$, and $0\leq a_i\leq r-2$, it is not difficult to
check that $0 \leq m_i \leq r-2$. By the WDVV equation
\eqref{eqwdvv}, we have
\begin{multline*}
\sum_{I\coprod J=\{4,\dots,s\}}\sum_{j=0}^{r-2}\bigg\langle
j,m_1,m_3,\prod_{i\in I}a_i\bigg\rangle\bigg\langle
r-2-j,m_2,m_4,\prod_{i\in J}a_i\bigg\rangle\\ = \sum_{{I\prod
J}=\{4, \dots, s\}}\sum_{j=0}^{r-2}\bigg\langle j,m_1,m_2,\prod_{i\in
I}a_i\bigg\rangle\bigg\langle r-2-j,m_3,m_4,\prod_{i\in
J}a_i\bigg\rangle.
\end{multline*}
Then Witten's formula follows from the inequalities $m_3 + m_4 >
r-2$ and $m_2 + m_4 > r-2$.

For the effectiveness of Witten's formula \eqref{eqwitten2}, it is
not difficult to prove that if $z'\geq y'\geq x'$ are the three
largest numbers in the index set
$\{x+y+z-(r-1),r-1-z,z,a_4,\dots,a_s\}$, then $r-1-z$ is not one of
$x',y',z'$ as long as $s\geq5$. On the other hand, each bracket in
the quadratic terms in the right hand side of \eqref{eqwitten2} has
strictly less than $s$ points.
\end{proof}

\vskip 30pt
\section{Formal pseudodifferential operators} \label{sectionpdo}

A formal pseudodifferential operator is an expression of the form
$$L=\sum_{i=-\infty}^N u_i(x) \partial^i,\quad \text{ where }
\partial=\frac{\partial}{\partial x}.$$
Its positive and negative parts are defined to be
$$L_+=\sum_{i=0}^N u_i(x) \partial^i,\qquad L_-=\sum_{i=-\infty}^{-1} u_i(x) \partial^i.$$

For $k\in \mathbb{Z}$, we define
$$\partial^k\cdot f=\sum_{j\geq0}\binom{k}{j}f^{(j)}\partial^{k-j}, \quad \text{ where } f^{(j)}=\frac{\partial^j f}{\partial x^j}.$$
We follow the usual convention that
$$\binom{-a-1}{b}=\binom{a+b}{b}(-1)^b, \qquad a,b\geq 0.$$
In particular, $\partial\cdot f=f'+f\partial$. Note that we reserve
the notation $\partial f$ for the derivative of $f$. It is straightforward to check that the set of all formal pseudodifferential
operators forms an associative algebra, denoted by $\Psi$DO.

The idea of fractional powers appeared in the work of Gel'fand and
Dikii \cite{GD}. It plays an important role in integrable systems
(cf.\ \cite{SW}). The following lemma is well-known.
\begin{lemma}\label{pdo3}
Recall the pseudodifferential operator $Q$ defined in \eqref{eqpdo7}
\begin{equation*}
Q=D^r+\sum_{i=0}^{r-2}\gamma_i(x)D^i.
\end{equation*}
 There exists a unique pseudodifferential operator of
the form
\begin{equation*}
Q^{1/r}=D+\sum_{i \geq 0}w_{-i}D^{-i},
\end{equation*}
whose $r$-th power is $Q$; and $w_0=0$.
\end{lemma}
\begin{proof}
Let $Q^{\frac{1}{r}}=D+w_0+w_{-1}D^{-1}+\cdots$. Then
$(Q^{\frac{1}{r}})^r=D^r+rw_0D^{r-1}+\cdots$. Since there is no
$D^{r-1}$ term on $Q$, we have $w_0=0$.   Thus we may write
\begin{equation*}
(Q^{\frac{1}{r}})^r=D^r+rw_{-1}D^{r-2}+(rw_{-2}+\frac{r(r-1)Dw_{-1}}{2})D^{r-3}+\cdots .
\end{equation*}

In general, we have $$rw_{-i}+p_i(w_{-1},\cdots
w_{-i+1})=\gamma_{r-1-i},$$ where $p_i$ is a differential polynomial
of its argument. So $w_{-i}$ can be uniquely determined recursively
as differential polynomials of $\gamma_i$.
\end{proof}

Fix $k\geq 1$.  Write
$$
Q^{k/r}=D^k+\sum_{i=0}^{k-2}\gamma_i^k
D^i+\sum_{i=1}^\infty\gamma_{-i}^k D^{-i}.
$$
Here we emphasize that throughout this paper, the superscript $k$ in
$\gamma_i^k$ never denotes a power. In particular, we have
$\gamma_i^r=\gamma_i$.

Since $Q^{(k+1)/r}=Q^{1/r}\cdot Q^{k/r}$, for $\ell\leq k-1$ we have
\begin{equation} \label{eqgamma}
\gamma_{\ell}^{k+1}=w_{\ell-k}+D\gamma_{\ell}^k+\gamma_{\ell-1}^k +
\sum_{j=1}^{k-2-\ell}w_{-j}\sum_{i=j+\ell}^{k-2}\binom{-j}{i-j-\ell}D^{i-j-\ell}\gamma_i^k.
\end{equation}
This identity can be used to determine $\gamma_{\ell}^{k+1}$
recursively as differential polynomials of $\{w_{-i}\}$.

\begin{lemma} \label{pdo1}
With the notation above, if we assign $w_{-i}^{(j)}=D^j w_{-i}$ the
weight $i+j+1$, then $\gamma_{\ell}^k$ is homogeneous of weight
$k-\ell$.
\end{lemma}
\begin{proof}
Since $\gamma_{\ell}^1=w_{\ell}$ is of weight $1-\ell$, the general
statement follows from the equation \eqref{eqgamma}.
\end{proof}

\begin{lemma}\label{pdo2}
Let $[w_{-i}^{(j)}]\gamma_{\ell}^k$ denote the coefficient of
$w_{-i}^{(j)}$ in $\gamma_{\ell}^k$. If $k\geq1,\ell\leq k-2$ and
$1\leq i\leq k-\ell-1$, then we have
\begin{equation}\label{eqpdo8}
[D^{k-\ell-i-1}w_{-i}]\gamma_{\ell}^k=\binom{k}{k-\ell-i}.
\end{equation}
In particular, $[w_{\ell-k+1}]\gamma_{\ell}^k=k$ and
$[Dw_{\ell-k+2}]\gamma_{\ell}^k=\frac{k(k-1)}{2}$.

\end{lemma}
\begin{proof}
When $k=1$, by definition, $\gamma_{\ell}^1=w_{\ell}$ for $\ell<0$.
The identity \eqref{eqpdo8} obviously holds in this case. So we
apply the recursive equation \eqref{eqgamma} and use induction on
$k$.

When $i=k-\ell-1$, we have
\begin{align*}
[w_{\ell-k+1}]\gamma_{\ell}^{k}&=1+[w_{\ell-k+1}]\gamma_{\ell-1}^{k-1}\\
&=1+k-1\\
&=k
\end{align*}
and similarly when $i<k-\ell-1$, we have
\begin{align*}
[D^{k-\ell-i-1}w_{-i}]\gamma_{\ell}^{k}&=[D^{k-\ell-i-2}w_{-i}]\gamma_{\ell}^{k-1}
+[D^{k-\ell-i-1}w_{-i}]\gamma_{\ell-1}^{k-1}\\
&=\binom{k-1}{k-\ell-i-1}+\binom{k-1}{k-\ell-i}\\
&=\binom{k}{k-\ell-i}
\end{align*}
as desired.
\end{proof}

\begin{lemma}[Witten, \cite{Wi2}]
\label{pdo4}With the above notation,
$\gamma_{-1}^{i+1}=\res (Q^{(i+1)/r})$, we can express coefficients
$\gamma_i$ of $Q$ as differential polynomials in
$\gamma_{-1}^{i+1},\,0\leq i\leq r-2$.
\end{lemma}

\begin{proof} By Lemmas \ref{pdo1} and \ref{pdo2}, we have
\begin{align}\label{eqpdo13}
\gamma_{-1}^{i+1}&=(i+1)w_{-1-i}+p_i(w_{-1},\cdots
w_{-i})\\&=\frac{(i+1)\gamma_{r-2-i}}{r}+p'_i(\gamma_{r-2},\cdots
\gamma_{r-1-i}),\nonumber
\end{align}
where $p_i$ and $p'_i$ are differential polynomials of their
arguments. Thus we can recursively express $\gamma_i$ as differential
polynomials in $\gamma_{-1}^{i+1},\ 0\leq i\leq r-2$.
\end{proof}

Denote by $P(\gamma_{\ell}^{k})$ the sum of monomials in $\gamma_{\ell}^{k}$
that does not contain derivatives of $w_{-i}$. Then we have
\begin{align}\label{eqpdo1}
P(\gamma_{\ell}^{k})&=[p^{k-\ell}] \left( 1+\sum_{i>0}w_{-i}p^{i+1} \right)^{k}\\
&=\res_{p=0}\frac{(1+\sum_{i>0}w_{-i}p^{i+1})^{k}}{p^{k-\ell+1}}.
\nonumber
\end{align}

Fix an integer $r\geq 2$. From the Gel'fand-Dikii equation
\eqref{eqwitten}, we have
\begin{align*}
\gamma_{-1}^{m+1}&=\res(Q^{\frac{m+1}{r}})=-\frac{m+1}{r}\langle\langle\tau_{0,0}\tau_{0,m}\rangle\rangle,
\qquad \text{for } 0\leq m\leq r-2\\
&=(m+1)w_{-m-1}+\cdots \nonumber
\end{align*}
 and
\begin{align} \label{eqpdo3}
\gamma_{-1}^{r+1}&=\res(Q^{1+\frac{1}{r}})=\frac{r+1}{r^2}\langle\langle\tau_{0,0}\tau_{1,0}\rangle\rangle\\
&=(r+1)w_{-r-1}+\frac{r(r+1)}{2}Dw_{-r}+\cdots. \nonumber
\end{align}

For the first time, we use the  fact that $Q$ is a differential
operator (i.e.\ $Q_-=0$), which implies that
\begin{equation} \label{eqpdo4}
0=\gamma_{-1}^{r}=r\cdot w_{-r}+\cdots \quad \quad \text{and}
\end{equation}
\begin{equation} \label{eqpdo5}
0=\gamma_{-2}^{r}=r\cdot w_{-r-1}+\cdots.
\end{equation}
The leading coefficients of the above equations come from Lemma
\ref{pdo2}.

We first substitute \eqref{eqpdo5} and then \eqref{eqpdo4} into
\eqref{eqpdo3} to eliminate $w_{-r-1}$ and $w_{-r}$ respectively.
Next we substitute $\gamma_{-1}^{r-1}, \gamma_{-1}^{r-2}, \dots,
\gamma_{-1}^1$ consecutively into \eqref{eqpdo3} to eliminate
$w_{-r+1}, w_{-r+2},\dots,w_{-1}$ successively. Then it is easy to
see that $\gamma_{-1}^{r+1}$ is now expressed in terms of
differential polynomials of $\gamma_{-1}^{m+1},\ 0\leq m\leq r-2$.
From now we on will use $S(\gamma_{-1}^{r+1})$ to denote this
differential polynomial in $\gamma_{-1}^{m+1},\ 0\leq m\leq r-2$
resulting from substitutions in $\gamma_{-1}^{r+1}$. We will keep
the notation $\gamma_{-1}^{r+1}$ for the differential polynomial
\eqref{eqpdo3} in $w_{-i}$.

If we use the notation
\begin{equation} \label{eqpdo9}
z_m^{(j)}=-\frac{r}{m+1}\cdot\frac{\partial^j\gamma_{-1}^{m+1}}{\partial
x^j} =\langle\langle\tau_{0,0}^{j+1}\tau_{0,m}\rangle\rangle,
\end{equation}
then we have the following structure theorem of formal
pseudodifferential operators.

\begin{theorem} \label{spin2} (As discussed above, we may regard
  $S(\gamma_{-1}^{r+1})$ as a differential polynomial in $z_m$.)
We have
$$\frac{r^2}{r+1}S(\gamma_{-1}^{r+1})=\frac{1}{2}\sum_{j=0}^{r-2}z_j
z_{r-2-j}+W_r(z),$$
where $W_r(z)$ represents the terms containing derivatives of some
$z_m$.
\end{theorem}
\begin{proof}
Since (16) is used to eliminate $w_{-r-1}$ in $\gamma_{-1}^{r+1}$,
it is not difficult to see that the identity of Theorem~\ref{spin2} is
equivalent to
$$\frac{r^2}{r+1}P(\gamma_{-1}^{r+1})-rP(\gamma_{-2}^{r})=
\frac{1}{2}\sum_{j=0}^{r-2}\frac{-r}{j+1}P(\gamma_{-1}^{j+1})\frac{-r}{r-1-j}P(\gamma_{-1}^{r-1-j}).$$

From equation \eqref{eqpdo1}, this is precisely the combinatorial
identity shown in the next proposition.
\end{proof}

\begin{proposition} \label{comb} Let $a_j$ be formal variables and
$$f(x)=1+\sum_{j=2}^\infty a_j x^j\in \mathbb C[[x]]$$ be a formal series satisfying $f(0)=1$ and $f'(0)=0$.
Then for any $n\geq 1$,
\begin{equation} \label{eqcomb}
\frac{[x^{n+2}]f^{n+1}}{n+1}=\frac12\sum_{j=1}^{n-1}\frac{[x^{j+1}]f^{j}}{j}\cdot\frac{[x^{n-j+1}]f^{n-j}}{n-j}
+\frac{[x^{n+2}]f^n}{n},
\end{equation}
where $[x^n]f^{k}$ denotes the coefficient of $x^n$ in the series
expansion of $f^{k}$.
\end{proposition}

The proof of Proposition \ref{comb} along with other interesting
equivalent formulations can be found in Appendix A.

\begin{example}
We illustrate the above procedure explicitly for $r=4$. Let
$Q^{1/4}=D+\sum_{i>0}w_{-i}D^{-i}$. Then
\begin{align*}
-\frac{1}{4}\langle\langle\tau_{0,0}\tau_{0,0}\rangle\rangle =\res(Q^{1/4})&=w_{-1},\\
-\frac{1}{2}\langle\langle\tau_{0,0}\tau_{0,1}\rangle\rangle=\res(Q^{2/4})&=2w_{-2}+Dw_{-1},\\
-\frac{3}{4}\langle\langle\tau_{0,0}\tau_{0,2}\rangle\rangle=\res(Q^{3/4})&=3w_{-3}+D^2w_{-1}+3Dw_{-2}+3w_{-1}^2.
\end{align*}

We also have
\begin{align*}
0=\res(Q)=&4w_{-4}+D^3w_{-1}+4D^2w_{-2}+6Dw_{-3}+6w_{-1}Dw_{-1}+12w_{-1}w_{-2},\\
0=\gamma_{-2}^4=&4w_{-5}+6Dw_{-4}+4D^2w_{-3}+D^3w_{-2}+6w_{-1}Dw_{-2}-(Dw_{-1})^2\\
&+12w_{-1}w_{-3}+6w_{-2}^2+4w_{-1}^3+2w_{-1}D^2w_{-1}.
\end{align*}

Substituting the above two groups of identities into
\begin{multline*}
\gamma_{-1}^5=\frac{5}{16}\langle\langle\tau_{0,0}\tau_{1,0}\rangle\rangle=\res(Q^{5/4})\\
=5w_{-5}+D^4w_{-1}+5D^3w_{-2}+10D^2w_{-3}+10Dw_{-4}+5(Dw_{-1})^2+10w_{-1}D^2w_{-1}\\
+10w_{-2}^2+10w_{-1}^3+20w_{-1}Dw_{-2}+10w_{-2}Dw_{-1}+20w_{-1}w_{-3}
\end{multline*}
and using $D=\frac{\sqrt{-1}}{2}\frac{\partial}{\partial x}$, we get
\begin{equation*}
\frac{16}{5}\gamma_{-1}^5=z_0z_2+\frac{1}{2}z_1^2+\frac{1}{4}z_2^{(2)}+\frac{1}{48}z_0z_0^{(2)}+\frac{1}{32}z_0'z_0'
+\frac{1}{480}z_0^{(4)}.
\end{equation*}
\end{example}
If we substitute the $z_m$ using equation \eqref{eqpdo9}, we get
exactly the recursion formula \eqref{eqspin4}.
\medskip

The universal differential polynomial $W_r(z)$ in
$z_0,\dots,z_{r-2}$ is particularly interesting in view of Theorem
\ref{spin2}. We present $W_r(z)$ for $2\leq r\leq6$ below:
\begin{gather*}
W_2(z)=\frac{1}{12}z_0^{(2)},\qquad
W_3(z)=\frac{1}{6}z_1^{(2)},\\
W_4(z)=\frac{1}{4}z_2^{(2)}+\frac{1}{48}z_0z_0^{(2)}+\frac{1}{32}z_0'z_0'
+\frac{1}{480}z_0^{(4)},\\
W_5(z)=\frac{1}{10}z_0'z_1'+\frac{1}{30}z_0z_1^{(2)}+\frac{1}{30}z_0^{(2)}z_1
+\frac{1}{3}z_3^{(2)}+\frac{1}{150}z_1^{(4)},\\
W_6(z)=\frac{5}{864}z_0^{(3)}z_0'+\frac{1}{144}z_0(z_0')^2+\frac{1}{8}z_2'z_0'+\frac{1}{24}z_0z_2^{(2)}+\frac{1}{432}z_0^2z_0^{(2)}
+\frac{1}{24}z_2 z_0^{(2)}
\\+\frac{1}{72}z_2^{(4)}+\frac{1}{9072}z_0^{(6)}+\frac{11}{2592}(z_0^{(2)})^2+\frac{1}{12}(z_1')^2+\frac{1}{18}z_1z_1^{(2)}
+\frac{1}{720}z_0z_0^{(4)}+\frac{5}{12}z_4^{(2)}.
\end{gather*}
We now study their coefficients.
\begin{proposition} \label{pdo6} We have
$[z_{r-2}^{(2)}]W_r(z)=\frac{r-1}{12}$.
\end{proposition}
\begin{proof} From equation \eqref{eqpdo9} and $D=\frac{\sqrt{-1}}{\sqrt{r}}\frac{\partial}{\partial x}$, we have
$$
\frac{\partial^2 z_{r-2}}{\partial x^2}=-\frac{r}{r-1}\cdot \frac{\partial^2 \gamma_{-1}^{r-1}}{\partial x^2}=\frac{r^2}{r-1}D^2\gamma_{-1}^{r-1}.
$$
So from Theorem \ref{spin2}, we get
\begin{equation} \label{eqpdo11}
[z_{r-2}^{(2)}]W_r(z)=\frac{r-1}{r+1}[D^2\gamma_{-1}^{r-1}]S(\gamma_{-1}^{r+1}).
\end{equation}

Recall that in $\gamma_{-1}^{r+1}=(r+1)w_{-r-1}+\frac{r(r+1)}{2}Dw_{-r}+\cdots$, we first substitute $w_{-r-1}$ using $\gamma_{-2}^r$ and
then substitute $w_{-r}$ using $\gamma_{-1}^r$, see equations \eqref{eqpdo4}, \eqref{eqpdo5}.  Then $\gamma_{-1}^{r+1}$ becomes a
differential polynomial in $w_{-1},\dots,w_{-r+1}$. We need to take care that when substituting $w_{-r-1}$ by $\gamma_{-2}^r$, a new term of $Dw_{-r}$
will appear.
With the above substitutions in mind and note that $\gamma_{-1}^{r-1}=(r-1)w_{-r+1}+\cdots$,
we may apply Lemma \ref{pdo2} to get
\begin{multline*}
\frac{r-1}{r+1}[D^2\gamma_{-1}^{r-1}]S(\gamma_{-1}^{r+1})=\frac{1}{r+1}[D^2w_{-r+1}]\left(\gamma_{-1}^{r+1}-(r+1)\cdot\frac{1}{r}\gamma_{-2}^{r}\right)\\
+\frac{1}{r+1}[D^2w_{-r+1}]\left(\left(\frac{r+1}{r}[Dw_{-r}]\gamma_{-2}^r-\frac{(r+1)r}{2}\right)\cdot\frac{1}{r}D\gamma_{-1}^r\right)\\
 =\frac{1}{r+1}\left(\frac{(r+1)r(r-1)}{3!}-\frac{r+1}{r}\cdot\frac{r(r-1)(r-2)}{3!}\right)\\
\qquad\qquad +\frac{1}{r+1}\left(\frac{r+1}{r}[Dw_{-r}]\gamma_{-2}^r-\frac{(r+1)r}{2}\right)\cdot\frac{1}{r}[Dw_{-r+1}]\gamma_{-1}^r\\
=\frac{r-1}{3}+\frac{-1}{2r}[Dw_{-r+1}]\gamma_{-1}^r
=\frac{r-1}{12}.
\end{multline*}
From \eqref{eqpdo11}, we get the desired result.
\end{proof}

\begin{corollary}[Witten, \cite{Wi2}]  We have the following
identity for $r$-spin numbers:
$$\langle\tau_{1,0}\rangle_1=\frac{r-1}{24}.$$
\end{corollary}
\begin{proof}
From the dilaton equation, $\langle\tau_{1,0}\tau_{0,0}\tau_{2,0}\rangle_1=2\langle\tau_{1,0}\rangle_1$. On the other hand, from
Theorems \ref{spin} and \ref{spin2}, we have
\begin{align*}
\langle\tau_{1,0}\tau_{0,0}\tau_{2,0}\rangle_1&=[z_{r-2}^{(2)}]W_r(z)\langle\tau_{0,0}^3\tau_{0,r-2}\tau_{2,0}\rangle_0\\
&=\frac{r-1}{12}.
\end{align*}
In the right hand side of the first equation, all other terms vanish
for dimensional reason (see \eqref{eqdim2}, \eqref{eqdim3}).
Hence $\langle\tau_{1,0}\rangle_1=\frac{r-1}{24}$.
\end{proof}

Proposition \ref{pdo6} generalizes as follows.

\begin{proposition}\label{pdo5} Suppose  $2\leq i\leq r$.
If  $i$ is odd, then
$[z_{r-i}^{(i)}]W_{r}(z)=0$.
If $i$ is even ($2k$, say), then
$$
[z_{r-2k}^{(2k)}]W_{r}(z)=\frac{(-1)^{k+1}(r+1-2k)}{r^k(r+1)}\binom{r+1}{2k}B_{2k},
$$
where $B_{2k}$ are Bernoulli numbers.
\end{proposition}
See Appendix B for a proof. By similar arguments, we
have the following fact, which means that the
genera in the right-hand side of \eqref{eqmain} are integers (see
\eqref{eqdim3}).   We omit the details.

\begin{proposition} The order of derivatives in each monomial of $W_r(z)$ is an even
number.
\end{proposition}

\vskip 30pt
\section{An algorithm for computing Witten's $r$-spin
numbers}\label{sectionalgorithm}
Let $\eta^{ij}=\delta_{i+j,r-2}$
and
$$\langle\langle\tau_{n_1,m_1}\dots\tau_{n_s,m_s}\rangle\rangle=\frac{\partial}{\partial t_{n_1,m_1}}\dots\frac{\partial}{\partial t_{n_s,m_s}}F(t_{0,0},t_{0,1},\dots).$$

The main result of this paper is the following
simple and effective recursion formula for computing all $r$-spin
intersection numbers.

\begin{theorem} \label{spin} For fixed $r\geq 2$, we have
\begin{equation} \label{eqmain}
\langle\langle\tau_{1,0}\tau_{0,0}\rangle\rangle_g=\frac{1}{2}\langle\langle\tau_{0,0}\tau_{0,m'}\rangle\rangle_{g'}
\eta^{m'm''}\langle\langle\tau_{0,m''}\tau_{0,0}\rangle\rangle_{g-g'}
+ \operatorname{Lower}(r),
\end{equation}
where $\operatorname{Lower}(r)$ is a explicit sum of products of
$\langle\langle\dots\rangle\rangle$ with genera strictly lower than
$g$.
\end{theorem}
\begin{proof} Since
$\frac{r^2}{r+1}\gamma_{-1}^{r+1}=\langle\langle\tau_{0,0}\tau_{1,0}\rangle\rangle$,
from Theorem \ref{spin2}, we need only prove that those monomials in
$W_r(z)$ must have genera strictly less than the left hand side.

Let us compare $\langle\langle\tau_{0,0}\tau_{1,0}\rangle\rangle_g$
and
$$\prod_{k}z_{i_k}^{(j_k)}=\prod_k\langle\langle\tau_{0,0}^{j_k+1}\tau_{0,i_k}\rangle\rangle_{g_k}.$$
Since the weight of $\gamma_{-1}^{r+1}$ is $r+2$ and the weight of
$z_m^{(j)}$ is $m+j+2$, we have
\begin{equation} \label{eqdim2}
\sum_{k} (i_k+j_k+2)=r+2.
\end{equation}
Combining with the dimensional constraints \eqref{eqdim}, we have
\begin{equation} \label{eqdim3}
(2r+2) \left( g-\sum_k g_k \right)=(r+1)\sum_{k} j_k.
\end{equation}
So $g=\sum_k g_k$ if and only if all $j_k=0$.
\end{proof}

\begin{remark}
Following a suggestion of Witten \cite[p.248]{Wi2}, Shadrin \cite{Sh} derived an expansion of $\langle\langle\tau_{n,m}\tau_{0,0}^2\rangle\rangle$ when $r=3$
and used it to compute some special $r$-spin numbers. Because of a
lack of an elegant structural description,
Shadrin's formula (and its generalization to higher $r$) results in a
much more complicated algorithm than \eqref{eqmain}.
\end{remark}

For example, when $r=3$, \eqref{eqmain} gives
\begin{equation}\label{eqspin3}
\langle\langle\tau_{1,0}\tau_{0,0}\rangle\rangle_g=\langle\langle\tau_{0,0}\tau_{0,1}\rangle\rangle_{g'}\langle\langle\tau_{0,0}^2\rangle\rangle_{g-g'}
+\frac{1}{6}\langle\langle\tau_{0,0}^3\tau_{0,1}\rangle\rangle_{g-1}.
\end{equation}
When $r=4$, we have
\begin{multline}\label{eqspin4}
\langle\langle\tau_{1,0}\tau_{0,0}\rangle\rangle_g=\langle\langle\tau_{0,0}\tau_{0,2}\rangle\rangle_{g'}\langle\langle\tau_{0,0}^2\rangle\rangle_{g-g'}
+\frac{1}{2}\langle\langle\tau_{0,0}\tau_{0,1}\rangle\rangle_{g'}\langle\langle\tau_{0,0}\tau_{0,1}\rangle\rangle_{g-g'}\\
+\frac{1}{4}\langle\langle\tau_{0,0}^3\tau_{0,2}\rangle\rangle_{g-1}+\frac{1}{48}\langle\langle\tau_{0,0}^2\rangle\rangle_{g'}\langle\langle\tau_{0,0}^4\rangle\rangle_{g-1-g'}
+\frac{1}{32}\langle\langle\tau_{0,0}^3\rangle\rangle_{g'}\langle\langle\tau_{0,0}^3\rangle\rangle_{g-1-g'}\\
+\frac{1}{480}\langle\langle\tau_{0,0}^6\rangle\rangle_{g-2}.
\end{multline}

Now we show how to use Theorem 1.1 to compute intersection numbers.
It consists of three steps.

{\em (i)}
When $g=0$, these intersection numbers can be computed by WDVV
equations, using Witten's algorithm \cite{Wi2}, as discussed in \S
\ref{sectionwitten}.

{\em (ii)}  Assume now that  $g\geq1$.
For an intersection number containing a puncture operator
$\langle\tau_{0,0}\tau_{n_1,m_1}\dots\tau_{n_s,m_s}\rangle_g$, we
have from Theorem 1.1 and the dilaton equation
\begin{multline} \label{eqalg}
(2g-1+s-a)\langle\tau_{0,0}\tau_{n_1,m_1}\dots\tau_{n_s,m_s}\rangle_g\\=\frac12\sum^{\sim
}_{\underline{s}=I\coprod J}\langle\tau_{0,0}\tau_{0,m'}\prod_{i\in
I}\tau_{n_i,m_i}\rangle_{g'}\eta^{m'm''}\langle\tau_{0,m''}\tau_{0,0}\prod_{i\in
J}\tau_{n_i,m_i}\rangle_{g-g'}+ \operatorname{Lower}(r)
\end{multline}
where $a=\#\{i\mid n_i=0\}$. Note that in the summation of the
right-hand side, we rule out the cases $I=\{i_1\}$ and $n_{i_1}=0$
or $J=\{i_1\}$ and $n_{i_1}=0$. Then the right hand side follows by
induction on genera or numbers of marked points.

{\em (iii)}
For any intersection number
$\langle\tau_{n_1,m_1}\dots\tau_{n_s,m_s}\rangle_g$ with $n_1\geq
n_2\geq\dots\geq n_s$, we apply the string equation first:
\begin{multline*}
\langle\tau_{n_1,m_1}\dots\tau_{n_s,m_s}\rangle_g=\langle\tau_{0,0}\tau_{n_1+1,m_1}\dots\tau_{n_s,m_s}\rangle_g
-\sum_{j=2}^s\langle\tau_{n_1+1,m_1}\tau_{n_j-1,m_j}\prod_{i\neq 1,j}\tau_{n_i,m_i}\rangle_g\\
\end{multline*}
The first term in the right hand side follows from step (ii) and the
second term follows by induction on the maximum descendent index.
This ends the algorithm.

The results of the above algorithm  agree with the  table of $r$-spin
numbers when $r=3$ and $4$ given in \cite{LX2}.  Some
$r$-spin numbers when $r=5$ are presented in Table \ref{tb1}.

\begin{table}[!htp]
\centering \caption{Witten's  $r$-spin numbers ($r=5$)}\label{tb1}
\begin{tabular}{|c|c||c|c||c|c|}

\hline $\langle\tau_{1,0}\rangle_1$ & $\frac{1}{6}$&
$\langle\tau_{0,2}\tau_{1,3}\rangle_1$ & $\frac{1}{60}$ &
 $\langle\tau_{0,1}\tau_{0,1}\tau_{2,3}\rangle_1$ & $\frac{1}{30}$

\\ \hline $\langle\tau_{3,2}\rangle_2$ & $\frac{11}{3600}$ & $\langle\tau_{0,3}\tau_{1,2}\rangle_1$ & $\frac{1}{60}$ &
$\langle\tau_{0,1}\tau_{0,2}\tau_{2,2}\rangle_1$ & $\frac{1}{20}$

\\ \hline $\langle\tau_{8,1}\rangle_4$ & $\frac{341}{25920000}$ & $\langle\tau_{0,1}\tau_{4,1}\rangle_2$ & $\frac{7}{1200}$
& $\langle\tau_{0,1}\tau_{0,3}\tau_{2,1}\rangle_1$ & $\frac{1}{20}$

\\ \hline $\langle\tau_{10,3}\rangle_5$ & $\frac{161}{777600000}$ & $\langle\tau_{0,2}\tau_{4,0}\rangle_2$ & $\frac{7}{1200}$
& $\langle\tau_{0,2}\tau_{0,2}\tau_{2,1}\rangle_1$ & $\frac{1}{15}$

\\ \hline $\langle\tau_{13,0}\rangle_6$ & $\frac{3397\times 10^{-6}}{93312}$ & $\langle\tau_{1,1}\tau_{3,1}\rangle_2$ & $\frac{17}{1200}$
& $\langle\tau_{0,2}\tau_{0,3}\tau_{2,0}\rangle_1$ & $\frac{1}{30}$

\\ \hline $\langle\tau_{15,2}\rangle_7$ & $\frac{3421\times 10^{-7}}{419904}$ & $\langle\tau_{1,2}\tau_{3,0}\rangle_2$ & $\frac{47}{3600}$
&$ \langle\tau_{0,1}\tau_{0,1}\tau_{5,0}\rangle_2$ &
$\frac{31}{3600}$

\\ \hline  $\langle\tau_{20,1}\rangle_{9}$ & $\frac{1670581\times 10^{-9}}{846526464}$
&  $\langle\tau_{2,0}\tau_{2,2}\rangle_2$ & $\frac{59}{3600}$ &
$\langle\tau_{0,1}\tau_{0,3}\tau_{4,3}\rangle_2$ & $\frac{7}{6000}$

\\ \hline  $\langle\tau_{22,3}\rangle_{10}$ & $\frac{2660573\times 10^{-12}}{1088391168}$
&  $\langle\tau_{2,1}\tau_{2,1}\rangle_2$ & $\frac{9}{400}$ &
$\langle\tau_{0,2}\tau_{0,2}\tau_{4,3}\rangle_2$ & $\frac{1}{500}$

\\ \hline  $\langle\tau_{25,0}\rangle_{11}$ & $\frac{21324511\times 10^{-12}}{5986151424}$
&  $\langle\tau_{3,2}\tau_{3,2}\rangle_3$ & $\frac{697}{324000}$ &
$\langle\tau_{0,2}\tau_{0,3}\tau_{4,2}\rangle_2$ & $\frac{23}{9000}$

\\ \hline  $\langle\tau_{27,2}\rangle_{12}$ & $\frac{87572287\times 10^{-14}}{13060694016}$
&  $\langle\tau_{3,1}\tau_{3,3}\rangle_3$ & $\frac{1111}{756000}$ &
$\langle\tau_{0,3}\tau_{0,3}\tau_{4,1}\rangle_2$ & $\frac{1}{500}$

\\ \hline  $\langle\tau_{32,1}\rangle_{14}$ & $\frac{7787064791\times 10^{-16}}{6582589784064}$
&  $\langle\tau_{2,1}\tau_{4,3}\rangle_3$ & $\frac{803}{756000}$ &
$\langle\tau_{0,1}\tau_{1,1}\tau_{4,0}\rangle_2$ & $\frac{17}{600}$

\\ \hline  $\langle\tau_{34,3}\rangle_{15}$ & $\frac{538156369\times 10^{-17}}{4231664861184}$
&  $\langle\tau_{2,2}\tau_{4,2}\rangle_3$ & $\frac{557}{324000}$ &
$\langle\tau_{1,1}\tau_{1,1}\tau_{3,0}\rangle_2$ & $\frac{41}{600}$

\\\hline
\end{tabular}
\end{table}

\noindent {\bf The Boussinesq hierarchy ($r=3$).}

For the remainder of this section, let $r=3$.  The $3$-KdV hierarchy is
also called the {\em Boussinesq hierarchy}. We see from \eqref{eqspin3},
\eqref{eqspin4} that compared with the recursive formula for
$3$-spin intersection numbers, the recursive formula for
$r$-spin numbers are much more complicated for $r \geq 4$.

The following closed formula holds for intersection numbers when
$r=3$.   This generalizes the special case $k=0$  obtained by Br\'{e}zin and
Hikami in \cite{BH}.
\begin{proposition} \label{threespin} Let $k\geq0$ and $0\leq j\leq1$. Then
$$
\langle\tau_{0,1}^k\tau_{\frac{8g+2k-5-j}{3},j}\rangle_g=\frac{1}{12^g g!}\frac{\Gamma(\frac{g+k+1}{3})}{\Gamma(\frac{2-j}{3})},
$$
where $\Gamma(z)$ is the gamma function.
\end{proposition}
\begin{proof} We first prove the identity in $g=0$ by induction on $k$, namely
$$\langle\tau_{0,1}^k\tau_{\frac{2k-5-j}{3},j}\rangle_0=\frac{\Gamma(\frac{k+1}{3})}{\Gamma(\frac{2-j}{3})}.$$
When $k=3,4,5$ respectively, we readily verify
$$\langle\tau_{0,1}^3\tau_{0,1}\rangle_0=\frac{1}{3},\qquad \langle\tau_{0,1}^4\tau_{1,0}\rangle_0=\frac{2}{3},\qquad
\langle\tau_{0,1}^5\tau_{1,2}\rangle_0=0.$$
Note the last identity is consistent with the fact that $\Gamma(z)$ has a simple pole at $z=0$.

Thus we may assume $k\geq6$.  We apply the genus zero topological
recursion relation \eqref{eqzero} to obtain:
\begin{align*}
\langle\tau_{0,1}^k\tau_{\frac{2k-5-j}{3},j}\rangle_0&=\sum_{i=0}^{k-2}\binom{k-2}{i}\langle\tau_{\frac{2k-8-j}{3},j}\tau_{0,0}\tau_{0,1}^i\rangle_0
\langle\tau_{0,1}^3\tau_{0,1}^{k-2-i}\rangle_0\\
&=(k-2)\langle\tau_{\frac{2k-8-j}{3},j}\tau_{0,0}\tau_{0,1}^{k-3}\rangle_0
\langle\tau_{0,1}^3\tau_{0,1}\rangle_0\\
&=\frac{k-2}{3}\cdot\frac{\Gamma(\frac{k-2}{3})}{\Gamma(\frac{2-j}{3})}\\
&=\frac{\Gamma(\frac{k+1}{3})}{\Gamma(\frac{2-j}{3})}.
\end{align*}
The second equation comes from dimensional constraints.
Thus we have proved Proposition~\ref{threespin} when $g=0$.

We next assume $g\geq1$ and proceed by induction on $g$.  We have
\begin{equation*}
\langle\tau_{1,0}\tau_{0,0}\tau_{0,1}^k\tau_{\frac{8g+2k-2-j}{3},j}\rangle_g=k\langle\tau_{0,0}\tau_{0,1}^k\tau_{\frac{8g+2k-2-j}{3},j}\rangle_g
\langle\tau_{0,0}^2\tau_{0,1}\rangle_0
+\frac{1}{6}\langle\tau_{0,0}^3\tau_{0,1}^{k+1}\tau_{\frac{8g+2k-2-j}{3},j}\rangle_{g-1}.
\end{equation*}

Applying the dilaton equation \eqref{eqdilaton} and the string equation
\eqref{eqstring} to the above identity and combining the first term in
the right hand side with the left hand side, we get
\begin{align*}
\langle\tau_{0,1}^k\tau_{\frac{8g+2k-5-j}{3},j}\rangle_g&=\frac{1}{12g}\langle\tau_{0,1}^{k+1}\tau_{\frac{8g+2k-11-j}{3},j}\rangle_{g-1}\\
&=\frac{1}{12g}\cdot\frac{1}{12^{g-1} (g-1)!}\frac{\Gamma \left(
    \frac{(g-1)+(k+1)+1}{3}
\right)}{\Gamma(\frac{2-j}{3})}\\
&=\frac{1}{12^g g!}\frac{\Gamma(\frac{g+k+1}{3})}{\Gamma(\frac{2-j}{3})}
\end{align*}
as desired.
\end{proof}

We now show that the $3$-spin numbers in genus zero in general do
not have clean closed formulas in contrast to the case of $r=2$. We
will compute intersection numbers of the form
$\langle\tau_{0,1}^k\tau_{2,0}^\ell\rangle_0$, which is nonzero only
if $k\equiv 1 \mod 3$ and $2k-3\ell=8$.

We will use the temporary notation
$a_m=\langle\tau_{0,1}^{3m+1}\tau_{2,0}^{2m-2}\rangle_0$, for $m\geq 1$.
By applying \eqref{eqalg} to
$$\langle\tau_{1,0}\tau_{0,0}\tau_{0,1}^{3m+1}\tau_{2,0}^{2m-1}\rangle_0=(2m-1)(5m-3)a_m
$$
and using the dilaton and string equations, it is not difficult to
obtain
\begin{multline}\label{eqzero2}
(2m-2)(2m-1)(5m-3)a_m\\
=\sum_{i=1}^{m-1}\binom{3m+1}{3i+1}\binom{2m-1}{2i}2i(2i-1)(5i-1)(5i-3)(2m-2i-1)(5m-5i-3)a_i
a_{m-i}.
\end{multline}
For example, we recursively find $a_1=\langle\tau_{0,1}^4\rangle_0=\frac13$,
$a_2=\frac{80}{9}$, $a_3=\frac{179200}{9}$,
$a_4=\frac{1281280000}{3}$.

To simplify the above equation, we substitute
$$b_m=\frac{(5m-3)a_m}{(3m+1)!(2m-2)!}.$$
For example, $b_1=\frac1{36},\, b_2=\frac{1}{126},\,
b_3=\frac{2}{792},\, b_4=\frac{85}{52488}$. Then \eqref{eqzero2}
becomes
\begin{equation}\label{eqzero3}
(2m-2)b_m=\sum_{i=1}^{m-1} (5i-1)(3m-3i+1)b_i b_{m-i}.
\end{equation}

In terms of the generating function $y(x)=\sum_{i=1}^\infty b_i
x^i$, we can rewrite \eqref{eqzero3} as
$$15x^2\left(\frac{dy}{dx}\right)^2+(2xy-2x)\frac{dy}{dx}-y^2+2y=0,$$
from which we get a first order ODE
$$\frac{dy}{dx}=\frac{1-y-\sqrt{1+16y^2-32y}}{15x}.$$
Integrating both sides of
$$\frac{15 dy}{1-y-\sqrt{1+16y^2-32y}}=\frac{dx}{x},$$
we get
\begin{multline*}
x=\exp\left(\int\frac{15 dy}{1-y-\sqrt{1+16y^2-32y}} +C\right)\\
=\exp\left(\ln y +\ln
36-8y-32y^2-\frac{992}{3}y^3-4864y^4+O(y^5)\right)\\
=36y-288y^2-5760y^4-92160y^5+O(y^6).
\end{multline*}
The constant of integration $C$ is uniquely determined by the
initial value $b_1=\frac1{36}$. Thus $b_i$ can also be computed
using the Lagrange inversion formula (see Lemma \ref{lag}).
$$b_i=\frac{1}{i}\res_{y=0}\left(\frac{1}{x(y)^i}\right),\quad i\geq
1.$$

We note that  the above derivation  becomes more difficult if we
instead use the genus zero topological recursion relation
\eqref{eqzero} to compute
$\langle\tau_{0,1}^k\tau_{2,0}^\ell\rangle_0$.

\vskip 30pt
\section{The Euler characteristic of $\mathcal
M_{g,1}$}\label{sectioneuler}

We now give a proof for Harer and Zagier's formula of the Euler
characteristic
of the moduli space of curves:

\begin{theorem}[Harer-Zagier \cite{HZ},
see also \cite{BH2, Ko, MP, No, Pe}]
Let $g\geq1$. Then\label{t:hz}
$$\chi(\mathcal M_{g,1})=-\frac{B_{2g}}{2g}.$$
For example, $\chi(M_{1,1})=-\frac{1}{12}, \chi(\mathcal
M_{2,1})=\frac{1}{120}, \chi(\mathcal M_{3,1})=-\frac{1}{252}.$
\end{theorem}

The early proofs of Harer-Zagier's formula \cite{Ko, MP, No, Pe} all exploit the cell decomposition of decorated moduli space in terms of Ribbon graphs. There is an intriguing fact from Witten's construction \cite{Wi3} that the $r\rightarrow -1$ limit of $r$-spin numbers actually gives $\chi(\mathcal M_{g,1})$. The main difficulty is to derive an explicit formula for the one-point $r$-spin numbers. This was obtained recently by Brezin and Hikami \cite{BH2} using rather complicated techniques from matrix integrals. We will give a proof using only properties of $\Psi$DO. The proof will conclude just after Lemma~\ref{euler2}.

We will use the case $s=1$, and general $r$.   Our discussion so far has assumed $r \geq 2$.  However, for
{\em any} $r$, there is a generalized Kontsevich (Airy) matrix model, and under
the limit $r \rightarrow -1$, the model gives a logarithmic potential
corresponding to the Penner matrix model, whose asymptotic expansion gives the generating function
 of the Euler characteristic of $\mathcal{M}_{g,1}$, \cite{Pe}.  (We do not
understand how to make  \cite[(3.55-3.57)]{Wi3} precise, so we instead
refer the reader to
\cite[\S 6]{BH} or \cite{Mu} for a complete discussion.)  Thus by taking $r =-1$ in our formulas (interpreted
as analytic continuation), we may compute $\chi(\mathcal{M}_{g,1})$,
as follows.
%
Setting  $s=1$ in \eqref{eqrspin}, we have
\begin{equation}\label{eqeu10}
\lim_{r\rightarrow -1}\langle\tau_{n,m}\rangle_g\Big|_{
m=0}=\chi(\mathcal M_{g,1}),
\end{equation}
where $\chi(\mathcal M_{g,1})$ is the orbifold Euler characteristic
of $\mathcal M_{g,1}$.
We now proceed to compute the left side of \eqref{eqeu10}, thereby
computing $\chi(\mathcal{M}_{g,1})$.

By the Gel'fand-Dikii equation \eqref{eqwitten}, in order to compute
$\langle\tau_{0,0}\tau_{n,m}\rangle_g$, we need to compute the
coefficient of $(D\gamma^{r-1}_{-1})^{2g}$ in $\res
(Q^{n+(m+1)/r})$. Note that
\begin{equation} \label{eqeu5}
\gamma^{r-1}_{-1}=-\frac{r-2}{r}\langle\langle\tau_{0,0}\tau_{0,r-2}\rangle\rangle.
\end{equation}

By Lemma \ref{pdo4} and \eqref{eqpdo13}, we know that when
expressing $\gamma_{i}$ ($0\leq i\leq r-2$) in terms of
$\gamma_{-1}^{i+1}$ ($0\leq i\leq r-2$), only $\gamma_0$ contains
the term $\gamma^{r-1}_{-1}$,  with
\begin{equation}\label{eqeu6}
\gamma_0=\frac{r}{r-2}
\gamma^{r-1}_{-1}+p(\gamma_{-1}^{1},\dots,\gamma_{-1}^{r-2}),
\end{equation}
where $p$ is a differential polynomial in its arguments.

If we replace $\gamma_0$ by $x$ and denote by $L=D^r+x$, it is not
difficult to see from \eqref{eqeu5} and \eqref{eqeu6} that
\begin{equation}\label{eqeu8}
\langle\tau_{0,0}\tau_{n,m}\rangle_g=\frac{(-1)^g
c_{n,m}}{r^g}\times\text{ the constant term in }\res
(L^{n+(m+1)/r}).
\end{equation}

There exists a pseudodifferential operator $K\in\Psi$DO of the form
$$
K=1+\sum^\infty_{i=1}b_i(x)D^{-i},
$$
such that $KLK^{-1}=D^r$.

We can determine $K$ by comparing the coefficients at both sides of
\begin{equation}\label{eqeu1}
KL=D^rK.
\end{equation}

The first few terms are
\begin{multline}\label{eqeu15}
K=\frac{x^2}{2r}D^{-(r-1)}+\frac{(1-r)x}{2r}D^{-r}+\frac{x^4}{8r^2}D^{-(2r-2)}+\frac{7(1-r)x^3}{12r^2}D^{-(2r-1)}
\\+\frac{(r-1)(7r-3)x^2}{8r^2}D^{-2r}+\frac{(1-r)(10r^2-3r-1)x}{24r^2}D^{-(2r+1)}+\cdots.
\end{multline}

In general, we have
$$
K=1+\sum^{\infty}_{u=1}\sum^{2u}_{i=1}b_{ur+u-i}D^{-(ur+u-i)},
$$
where $b_{ur+u-i}=a_{u,i}x^i$ with $a_{u,i}$ rational functions of
$r$. In particular, from \eqref{eqeu15}, we have
$$
a_{1,1}=\frac{1-r}{2r},\quad a_{1,2}=\frac{1}{2r}.
$$

Given $u\geq1$ and $1\leq i\leq 2u$, if we equate the coefficient of
$D^{-u+i-1-(u-1)r}$ in \eqref{eqeu1}, we get
\begin{equation}\label{eq7}
a_{u-1,i-2}+\big(i-u-(u-1)r\big)a_{u-1,i-1}=\sum_{k=0}^{2u-i}\binom{r}{k+1}\prod_{j=0}^k(i+j)\cdot
a_{u,i+k}.
\end{equation}

By a tedious but straightforward calculation, we find the recursion
\begin{equation}\label{eqeu2}
i!\,a_{u,i}=\sum_{j=0}^{2u-2}a_{u-1,j}\big((j+1)!\,
s_{j+2-i}+(j-(u-1)(r+1))j!\, s_{j+1-i}\big),
\end{equation}
where $s_k$ is the coefficient of $x^k$ in
$$
\frac{x}{(1+x)^r-1}=\frac{1}{r+\binom{r}{2}x+\binom{r}{3}x^2+\cdots}.
$$

For convenience, let $e_{u,i}=i!\,a_{u,i}$.  Then \eqref{eqeu2}
becomes
\begin{equation} \label{eqeu3}
e_{u,i}=\sum_{j=1}^{2u-2}e_{u-1,j}\big((j+1)s_{j+2-i}+(j-(u-1)(r+1))s_{j+1-i}\big),
\end{equation}
with initial values
$$e_{1,1}=s_1=\frac{1-r}{2r},\qquad
e_{1,2}=s_0=\frac{1}{r}.$$ From the recursion, we see $e_{u,i}$ is
nonzero only when $1\leq i\leq 2u$.

\begin{proposition} \label{onept2} Let $g\geq0$. We have the following
formula for one-point $r$-spin numbers
\begin{equation}\label{eqeu14}
\langle\tau_{n,m}\rangle_{g}=\frac{(-1)^g\Gamma(-2g-\frac{2g-1}{r})}{r^g
\Gamma(1-\frac{m+1}{r})}E_{2g},
\end{equation}
where
$$
E_u=\sum_{i=1}^{2u}\binom{u(r+1)-1}{i}e_{u,i}.
$$
\end{proposition}
\begin{proof}
Since $L=K^{-1}D^r K$, we have $L^{n+1+(m+1)/r}=K^{-1}D^{(n+1)r+m+1}
K$. From
$\langle\tau_{n,m}\rangle_{g}=\langle\tau_{0,0}\tau_{n+1,m}\rangle_{g}$
and $(n+1)r+m+1=2g(r+1)-1$, it is not difficult to see that the
constant term in $\res (L^{n+1+(m+1)/r})$ equals the constant term
in $\res D^{(n+1)r+m+1} K$, which is $E_{2g}$. Finally
\eqref{eqeu14} follows from \eqref{eqeu8} and
\begin{multline*}
c_{n+1,m}=\frac{(-1)^{n+1}
r^{n+2}}{(m+1)(r+m+1)\cdots((n+1)r+m+1)}\\
=\frac{\Gamma(-n-1-\frac{m+1}{r})}{\Gamma(1-\frac{m+1}{r})}
=\frac{\Gamma(-2g-\frac{2g-1}{r})}{\Gamma(1-\frac{m+1}{r})}.
\end{multline*}
\end{proof}

Setting $m=0$ and taking $r\rightarrow -1$ in the right-hand side of
\eqref{eqeu14}, we get
$$
\lim_{r\rightarrow -1}\Gamma \left( -2g-\frac{2g-1}{r} \right)E_{2g},
$$
which is computed by applying L'H\^opital's Rule to the following
Lemma.
\begin{lemma} \label{euler2}
For any integer $u \geq 1$, we have
$$
\lim_{r\rightarrow -1}\Gamma(-u-\frac{u-1}{r})E_u =-\frac{B_u}{u}.
$$
\end{lemma}

\begin{proof}
Since the residue of $\Gamma(z)$ at $z=-1$ is $-1$, we have
\begin{equation}\label{eq5}
\lim_{r\rightarrow
-1}\frac{d}{dr}\left(\frac{1}{\Gamma(-u-\frac{u-1}{r})}\right)=1-u.
\end{equation}
We also have
\begin{equation}\label{eq6}
\frac{d}{dr}\Big{|}_{r=-1}\binom{u(r+1)-1}{i}=(-1)^{i+1}uH_i,
\end{equation}
where $H_i=\sum_{1\leq k\leq i}\frac{1}{k}$ is the $i$th harmonic number.

Setting $i=1$ in \eqref{eq7}, we get
$$
0=\sum^{2u}_{k=1}\binom{r}{k}e_{u,k},
$$
which, after taking derivative with respect to  $r$, becomes
\begin{equation}\label{eq8}
0=\sum^{2u}_{k=1}\Big((-1)^{k+1}H_k e_{u,k}(-1)+(-1)^k
e'_{u,k}(-1) \Big).
\end{equation}

From \eqref{eq6} and \eqref{eq8}, we have
\begin{align}\label{eq9}
\lim_{r\rightarrow -1}E_u&=\sum^{2u}_{k=1}\Big((-1)^{k+1}uH_k
e_{u,k}(-1)+(-1)^k e'_{u,k}(-1) \Big)\\
&=(u-1)\sum^{2u}_{k=1}(-1)^{k+1}uH_k e_{u,k}(-1).\nonumber
\end{align}

By \eqref{eq5} and \eqref{eq9}, we see that Lemma~\ref{euler2} is
equivalent to
\begin{equation}\label{eq11}
\sum^{2u}_{k=1}(-1)^{k+1}H_k e_{u,k}(-1)=\frac{B_u}{u},\qquad
u\geq 2.
\end{equation}

Setting $r=-1$ in \eqref{eqeu3}, we get
\begin{equation}\label{eq12}
e_{u,i}(-1)=(1-i)e_{u-1,i-2}(-1)+(1-2i)e_{u-1,i-1}(-1)-i e_{u-1,i}(-1).
\end{equation}
Here we use $s_0(-1)=s_1(-1)=-1$ and $s_k(-1)=0$, $k>1$.

If we substitute \eqref{eq12} into \eqref{eq11}, we get
\begin{equation}\label{eq13}
\sum^{2u-2}_{k=1} \frac{(-1)^{k}}{(k+1)(k+2)}
e_{u-1,k}(-1)=\frac{B_{u}}{u},\quad u\geq 2.
\end{equation}

By substituting \eqref{eq12} successively into \eqref{eq13}, we get
\begin{equation}\label{eq14}
\sum^{2(u-j)}_{k=0}(-1)^{k+1} f_j(k)
e_{u-j,k}(-1)=\frac{B_{u}}{u},\quad 1\leq j\leq u,
\end{equation}
where $f_j,\,j\geq1$ are given by the recursion
\begin{equation}\label{eq15}
f_{j+1}(k)=-(k+1)f_j(k+2)+(2k+1)f_j(k+1)-k f_j(k)
\end{equation}
starting with $f_1(k)=\frac{-1}{(k+1)(k+2)}$.

Let $e_{0,i}=\delta_{0i}$, which is compatible with the recursion
\eqref{eq12}.  Then \eqref{eq13} (hence \eqref{eq11}) is equivalent to
\begin{equation}\label{eq16}
f_u(0)=\frac{-B_{u}}{u},\quad u\geq 2.
\end{equation}
We leave the proof to the Appendix C.
\end{proof}

From \eqref{eqeu10}, \eqref{eqeu14} and Lemma~\ref{euler2}, we recover
the Harer-Zagier formula (Theorem~\ref{t:hz}).

Now we make a connection to the matrix integral approach of
Br\'ezin and Hikami \cite{BH2}. First we note that a refined
argument in the proof of Lemma \ref{euler2} will give the following
identity of generating functions.
\begin{multline} \label{eqgen}
\Gamma\left(\frac 1r\right)+\sum_{u=1}^\infty E_u
\Gamma\left(-u-\frac{u-1}{r}\right)y^u
\\=r\int_{0}^\infty  \exp\left(-\frac{1}{(r+1)y}\left[\left(x+\frac{y}{2}\right)^{r+1}-\left(x-\frac{y}{2}\right)^{r+1}\right]\right) dx .
\end{multline}
The integral expression of the right-hand side appeared in
\cite{BH2}.

Consider the semigroup $\mathbb{N}^\infty$ of sequences ${\bold
d}=(d_1,d_2,\dots)$ where $d_i$ are nonnegative integers and $d_i=0$
for sufficiently large $i$.
For $\bold d \in \mathbb{N}^\infty$, we define
\begin{equation} \label{eqeu7}
|\bold d|:=\sum_{i\geq 1}i d_i,\quad ||\bold
d||:=\sum_{i\geq1}d_i,\quad \bold d!:=\prod_{i\geq1}d_i!.
\end{equation}

In the following exposition, we set
$t_i=-\binom{r}{2i}/((2i+1)4^i)$, $i\geq1$.
\begin{proposition} \label{onept} Let $g\geq0$. We have the following
closed formula for one-point $r$-spin numbers
$$
\langle\tau_{n,m}\rangle_{g}=\frac{(-1)^g}{r^g\Gamma \left(
    1-\frac{1+m}{r} \right)}\sum_{|\bold
d|=g}\Gamma \left(||\bold d||-\frac{2g-1}{r} \right)\frac{\prod_{i\geq
1}t_i^{d_i}}{\bold d!}.
$$
\end{proposition}
\begin{proof} By expanding the right-hand side of \eqref{eqgen}
$$
-\frac{1}{(r+1)y}\left( \left(x+\frac{y}{2}\right)^{r+1}-\left(x-\frac{y}{2}\right)^{r+1}\right)
=-x^r-\sum_{i\geq1}t_i y^{2i}x^{r-2i}
$$ and using
$$
\Gamma(z)=r\int_0^\infty x^{rz-1}\exp(-x^r)dx,
$$
we see that Proposition~\ref{onept} follows from \eqref{eqeu14}.
\end{proof}

\begin{corollary} Let $k\geq 0$ and $g\geq0$. We have the following
closed formula for $r$-spin numbers\label{c}
$$\langle\tau_{0,1}^k\tau_{n,m}\rangle_{g}=\frac{(-1)^g}{r^g\Gamma(1-\frac{1+m}{r})}\sum_{|\bold
d|=g}\Gamma \left( ||\bold d||-\frac{2g-k-1}{r} \right) \frac{\prod_{i\geq
1}t_i^{d_i}}{\bold d!}.
$$
\end{corollary}

 This  follows from the same inductive
argument as in the proof of Proposition \ref{threespin}.

\begin{remark} For $r$-spin numbers, we do not have the analogue of the divisor equation as in
Gromov-Witten theory. But the identity of Proposition~\ref{onept} suggests
that some form of the ``divisor equation'' may still exist for $r$-spin
numbers.\label{rem}
\end{remark}

\begin{corollary}\label{euler} Let $g\geq1$. Setting $r=-1$, $m=0$ in the right-hand side of Proposition~\ref{onept}, we
get
$$
\sum_{|\bold d|=g}\Gamma(||\bold d||+2g-1)\frac{(-1)^{||d||}}{\bold
d!\prod_{i\geq1}((2i+1)4^i)^{d_i}}=-\frac{B_{2g}}{2g}.
$$
\end{corollary}
\begin{proof} We follow the method used \cite[\S 3]{BH2}.

Letting $r\rightarrow -1$ in the right-hand of \eqref{eqgen} and
applying L'H\^opital's Rule, we get
$$
RHS=-\int_0^\infty\left(\frac{x-\frac{y}{2}}{x+\frac{y}{2}}\right)^{1/y}dx,
$$
which is the generating function of the left-hand side of
Corollary~\ref{euler}.

Making the change of variables
$$
\frac{x-\frac{y}{2}}{x+\frac{y}{2}}=e^{-z},\qquad\text{ i.e. }\quad
x=\frac{y}{2}\left(\frac{1+e^z}{1-e^z}\right).
$$
we have
\begin{align}
RHS&=-\int_0^\infty e^{-z/y}\frac{-ye^{-z}}{(1-e^{-z})^2}dz\\
&=-\int_0^\infty e^{-z/y}\frac{1}{1-e^{-z}}dz\label{eqeu12}\\
&=-y\int_0^\infty e^{-t}dt\frac{1}{1-e^{-yt}}\\
&=-\sum_{k=1}^\infty \frac{B_k}{k}y^k.\label{eqeu13}
\end{align}
Here \eqref{eqeu12} follows from
$$
\frac{d}{dz}\left(\frac{e^{-z/y}}{1-e^{-z}}\right)=\frac{-e^{-z/y}}{y(1-e^{-z})}+\frac{-e^{-z}e^{-z/y}}{(1-e^{-z})^2}
$$
and \eqref{eqeu13} follows from
$$
\frac{1}{1-e^{-t}}=\sum_{k=0}^\infty B_k\frac{t^{k-1}}{k!}.
$$
This completes the proof.
\end{proof}

\vskip 30pt
\section{Small phase space in genus zero}\label{sectionzero}
In this section, we extend Witten's exposition in \cite{Wi2}. We first
reorganize Witten's argument and highlight important relevant results
of Witten  for the reader's convenience. We
then prove a full series expansion formula for the Landau-Ginzburg
potential $W(p, x)$ in the small phase space $(t_{n, m}=0, n>0)$ of
genus zero.

For dimensional reasons (equation \eqref{eqdim}), a primary
intersection number $\langle\tau_{0,m_1}\cdots\tau_{0,m_s}\rangle_g$
can be nonzero only when $g=0$. Furthermore, for each $r$, there are
only finite number of nonzero primary intersection numbers
$\langle\tau_{0,m_1}\cdots\tau_{0,m_s}\rangle_0$, since we have
$r(s-2)-2= m_1+\cdots+m_s\leq (r-2)s$ (so in particular $s\leq r+1$).

As observed by Witten, the genus zero Gel'fand-Dikii equation is
obtained by replacing the differential operator $Q$ by a function
$$W(p,x)=p^r+\sum_{i=0}^{r-2}u_i(x)p^i$$
and replacing commutators by Poisson brackets $$\{A,
B\}=\frac{\partial A}{\partial p}\frac{\partial B}{\partial x} -
\frac{\partial A}{\partial x}\frac{\partial B}{\partial p}.$$

So in genus zero, the Gel'fand-Dikii equations reduce to
$$\frac{\partial W}{\partial t_{n,m}}=
\frac{c_{n,m}}{r}\{W_+^{n+(m+1)/r},W\},$$ where $c_{n,m}$ is the
same constant defined in \S \ref{s:introduction}

\begin{lemma}[Witten, \cite{Wi2}]
\label{witten1}
 On the small phase space, we have
\begin{equation*}
\frac{\partial F}{\partial t_m}=\frac{r^2}{(m+1)(r+m+1)} \res
(W^{1+(m+1)/r}).
\end{equation*}
\end{lemma}

\begin{proof}
A special case of Witten's conjecture is
\begin{equation*}
\frac{\partial^2 F}{\partial t_{0,0}\partial t_{1, m}}=
\frac{r^2}{(m+1)(r+m+1)} \res (W^{1+(m+1)/r}).
\end{equation*}
The string equation implies that on small phase space, we actually
have
\begin{equation*}
\frac{\partial^2 F}{\partial t_{0.0}\partial t_{1, m}}=
\frac{\partial F}{\partial t_{0, m}}.
\end{equation*}
The desired equation follows.
\end{proof}

Below, all of our computations will be done entirely on the small
phase space $(t_{n,m}=0, n>0)$ and we set $t_m=t_{0,m}$ and
$\tau_m=\tau_{0, m}$.
\begin{lemma}[Witten, \cite{Wi2}]
\label{witten2}
For $0 \leq m\leq r-2$, we have
\begin{equation*}
\frac{\partial W}{\partial t_m}=-
\frac{1}{m+1}\frac{\partial}{\partial p} W_+^{(m+1)/r}.
\end{equation*}
\end{lemma}

\begin{proof}
A special case of Witten's conjecture is
\begin{equation}\label{eqpdo12}
\frac{\partial^2 F}{\partial t_0 \partial t_m}= \frac{-r}{m+1} \res
(W^{(m+1)/r}).
\end{equation}

By \eqref{eqpdo13} in the proof of Lemma \ref{pdo4} (the
differential polynomials $p,p'$ there should be replaced by plain
polynomials of their arguments), we can use equation \eqref{eqpdo12}
to express the coefficients $u_i$ of $W$ as differential polynomials
in $\partial^2 F/\partial t_0
\partial t_m$.   Hence $W$ can be regarded as a function in
$p,t_0,\dots,t_m$. If we set all $t_{m}=0$, then the left hand side
of \eqref{eqpdo12} is obviously zero for dimensional reasons, so all
$u_i=0$ by Lemma \ref{pdo4}. Thus $W=p^r$ when all $t_{m}=0$. We  then
get
the constant term of $W$.

Differentiating \eqref{eqpdo12} with respect to $x=t_0$, we get
\begin{equation}\label{eqpdo14}
\delta_{m,r-2}=-\frac{r}{m+1}\frac{\partial}{\partial
x}\res(W^{(m+1)/r}).
\end{equation}
Since $\partial W/\partial x$ is a polynomial in $p$ of degree at most
 $r-2$, if this polynomial is of degree $k$, then from \eqref{eqpdo13}
in the proof of Lemma \ref{pdo4}, the right hand side of
\eqref{eqpdo14} is non-zero for $m=r-2-k$. Thus $k=0$ and
\begin{equation*}
\frac{\partial W}{\partial x}=\frac{\partial u_0}{\partial
x}=\frac{r}{r-1}\frac{\partial}{\partial
x}\res(W^{\frac{r-1}{r}})=-1.
\end{equation*}
From this and $\partial u_i/\partial x=0$ when $1\leq i\leq r-2$, we
have for $0\leq m\leq r-2$,
\begin{equation*}
\frac{\partial}{\partial x}W_+^{(m+1)/r}=0,
\end{equation*}
since the coefficients of $W_+^{(m+1)/r}$ do not contain $u_0$.
This follows from a weight count, since by our convention (see Lemma
\ref{pdo1}), $W_+^{(m+1)/r}$ is homogeneous of weight $m+1\leq r-1$,
while the weight of $u_i$ is $r-i$.
\end{proof}

\begin{theorem}[Witten, \cite{Wi2}]
\label{witten3}
 For $0 \leq m \leq r-2$, define
$\phi_m=-\frac{\partial W}{\partial t_m}$.  Then
\begin{equation*}
\frac{\partial^3 F}{\partial t_j\partial t_m\partial t_s}=r\cdot
\res\left\{\frac{\phi_j \phi_m \phi_s}{\partial_p W}\right\}.
\end{equation*}
\end{theorem}

Now we can state our new results:   the full series expansion for $W$ in
$t_0,\dots,t_m$, extending Witten's computation up to linear terms
\cite{Wi2}.

\begin{theorem}\label{witten4}
We have the following series expansion for $W$:
\begin{multline*}
W=p^r+\sum_{k=0}^{r-2}p^k \sum_{n=1}^\infty \frac{(-1)^n}{n!\cdot
r^{n-1}}\sum_{v_1+\cdots +v_n=(n-1)r+k}\frac{(k+n-1)!}{k!}
t_{v_1}\cdots t_{v_n}\\
=p^r+\sum_{k=0}^{r-2}p^k\left(-t_k+\frac{1}{2!\cdot
r}\sum_{u+v=r+k}(k+1)t_u t_v \right.\\
\left.-\frac{1}{3!\cdot
r^2}\sum_{u+v+w=2r+k}(k+1)(k+2)t_u t_v t_w+\cdots\right)
\end{multline*}
\end{theorem}

\begin{proof} We can compute the degree $n$ term of $W_+^{(m+1)/r}$
  from terms of $W$ up to degree $n$. Then we use Lemma \ref{witten2}
  to compute the degree $n+1$ term of $W$ from the degree $n$ term of
  $W_+^{(m+1)/r}$.
\begin{multline*}
W_+^{(m+1)/r}=p^{m+1}- \frac{m+1}{r}\sum_{u\geq r-m-1} t_u
p^{m+u-r+1}\\+\frac{m+1}{2!\cdot r^2}\sum_{u+v\geq
2r-m-1}(m+u+v+2-2r) t_u t_v p^{m+u+v-2r+1}+\cdots\\=
p^{m+1}+(m+1)\sum_{n=1}^\infty \frac{(-1)^n}{n!\cdot
r^n}\sum_{v_1+\cdots +v_n\geq nr-m-1} \frac{(m+ \sum\limits_{i=1}^n
v_i+n-nr)!}{(m+\sum\limits_{i=1}^n v_i+1-nr)!}\\\times t_{v_i}\cdots
t_{v_n} p^{m+v_1+\cdots +v_n-nr+1}
\end{multline*}
Thus the theorem can be proved inductively.
\end{proof}

\begin{corollary} \label{witten5}
Let $0\leq m\leq r-2$. The series expansion for $\phi_m$ is
\begin{multline*}
\phi_m=-\frac{\partial W}{\partial t_m}=p^m-\sum_{u\geq r-m}
\frac{m+u+1-r}{r} t_u p^{m+u-r}\\+\frac{1}{2!\cdot r^2}\sum_{u+v\geq
2r-m}(u+v+m+1-2r)(u+v+m+2-2r) t_u t_v p^{m+u+v-2r}+\cdots\\=
p^m+\sum_{n=1}^\infty \frac{(-1)^n}{n!\cdot r^n}\sum_{v_1+\cdots
+v_n\geq nr-m} \frac{(m+ \sum\limits_{i=1}^n
v_i+n-nr)!}{(m+\sum\limits_{i=1}^n v_i-nr)!} t_{v_i}\cdots t_{v_n}
p^{m+v_1+\cdots +v_n-nr}
\end{multline*}
\end{corollary}
\begin{proof}
This follows from the definition of $\phi_m$ and a direct
computation.
\end{proof}

In \cite{DVV}, Dijkgraaf, Verlinde and Verlinde give  a closed
formula of $\phi_m$ in terms of the determinant of matrices.
Presumably their formula is equivalent to ours, although we have not checked the details.

In view of Lemma \ref{witten1}, it would also be interesting to have
a series expansion for $W^{\frac{m+1}{r}}$. Here we write out terms
up to degree $3$. Let $\theta(x)$ be the Heaviside function that is $1$ for
$x\geq0$ and $0$ for $x<0$.
\begin{multline}\label{eqpdo2}
W^{\frac{m+1}{r}}=p^{m+1}-\frac{m+1}{r} \sum_u t_u p^{m+1+u-r}
\\ +\frac{m+1}{2!\cdot r^2}\sum_{u,v}((m+1-r)+(u+v-r+1)\theta(u+v-r))
t_ut_vp^{m+u+v+1-2r}\\-\frac{(m+1)}{3!\cdot
r^3}\sum_{u,v,w}((m+1-r)(m+1-2r)+(m+1-r)(u+v+1-r)\theta(u+v-r)\\
+(m+1-r)(u+w+1-r)\theta(u+w-r)+(m+1-r)(v+w+1-r)\theta(v+w-r)\\
+(u+v+w-2r+1)(u+v+w-2r+2)\theta(u+v+w-2r))t_u t_v t_w
p^{m+u+v+w+1-3r}\\+\cdots
\end{multline}

By Lemma \ref{witten1}, we can use the degree $3$ term of the above
expansion to get a formula for $4$-point correlation functions.

\begin{corollary}\label{witten6}
\begin{multline*}
\langle \tau_m \tau_u \tau_v \tau_w \rangle
=\frac{1}{r}(r-m-1-(u+v-r+1)\theta(u+v-r)\\
-(u+w-r+1)\theta(u+w-r)-(v+w-r+1)\theta(v+w-r)).
\end{multline*}
\end{corollary}
\begin{proof}
Replace $m$ by $m+r$ in the expansion of $W^{\frac{m+1}{r}}$ and
take the coefficient of $p^{-1}$.  We get the desired result from
Lemma \ref{witten1}.
\end{proof}

The formula in Corollary \ref{witten6} is slightly different with
Witten's formula \cite[(3.3.36)]{Wi2}
\begin{multline*}
\langle \tau_m \tau_u \tau_v \tau_w \rangle
=\frac{1}{r}(m-(m+u-r+1)\theta(m+u-r)\\
-(m+v-r+1)\theta(m+v-r)-(m+w-r+1)\theta(m+w-r)),
\end{multline*}
but it is not difficult to prove that they are both equivalent to
$\langle\tau_{a_1}\tau_{a_2}\tau_{a_3}\tau_{a_4}\rangle=\frac{1}{r}\cdot
\min(a_i,r-1-a_i)$.

Our motivation in studying $r$-spin numbers on the small phase
space is to prove the following conjectural properties of these
numbers.

\begin{conjecture}
In the small phase space, we have\label{con}
\begin{enumerate}
\item[i)]{\rm(Integrality)} $\frac{r^{s-3}}{(s-3)!}\langle \tau_{m_1}\cdots \tau_{m_s}\rangle \in \mathbb
Z$.
\item[ii)] {\rm(Vanishing)} If $m_i<s-3$ for some $1\leq i\leq s$, then
$\langle \tau_{m_1}\cdots \tau_{m_s}\rangle=0$.
\item[iii)] {\rm(Multinomial distribution)} If $m_1 > m_2$, then $\langle \tau_{m_1-1} \tau_{m_2+1} \tau_{m_3}\cdots \tau_{m_s}\rangle \geq \langle \tau_{m_1} \tau_{m_2}\cdots
\tau_{m_s}\rangle$.
\end{enumerate}
\end{conjecture}

We have verified this conjecture in low genus or when $s$ is small. One can even prove $r^{s-3}\langle \tau_{m_1}\cdots \tau_{m_s}\rangle \in \mathbb
Z$ by extending \eqref{eqpdo2}. However, the
combinatorial difficulty for the general case is still considerable,
even though we can write down explicit formulae for general
$s$-point correlation functions $\langle \tau_{m_1}\cdots
\tau_{m_s}\rangle$ via Theorem \ref{witten4}, Corollary
\ref{witten5} and Witten's Theorem \ref{witten3}.

See \cite{LX1} for more on denominators and multinomial-type properties of
intersection numbers.

Part of our motivation comes from Gromov-Witten invariants of
$\mathbb{CP}^n$ on the small phase space. For fixed $n\geq1$ and
$d\geq 0$, consider Gromov-Witten invariants of $\mathbb{CP}^n$ on the
small phase
space in genus zero
$$\langle m_1,\cdots,m_s\rangle:=\langle
\tau_{0,m_1},\cdots,\tau_{0,m_s}\rangle_{0,d}^{\mathbb{CP}^n},$$
which is nonzero (for dimensional reasons) only when
\begin{equation}\sum_{i=1}^s m_i=n+(n+1)d+s-3.\label{eq:dec6}\end{equation}

The following properties are analogues of corresponding statements in Conjecture~\ref{con}:
\begin{enumerate}
\item[i)]{\rm(Integrality)} $\langle m_1,\cdots ,m_s\rangle \in \mathbb
N_{\geq0}$.
\item[ii)] {\rm(Vanishing)} If $d>1$ and $s=d+2$, then $\langle m_1,\cdots ,m_s\rangle=0$.
\item[iii)] {\rm(Multinomial distribution)} If $m_1 > m_2$, then $$\langle m_1-1, m_2+1, m_3, \cdots, m_s\rangle \geq \langle m_1,
m_2,\cdots, m_s\rangle.$$
\end{enumerate}

The integrality (i) is clear, since genus zero Gromov-Witten
invariants of $\mathbb{CP}^n$ are intersections on a scheme, and hence
integral.  We conjecture the multinomial distribution (iii),
based on numerical evidence.    We now prove (ii).

\begin{proposition}    With the notation above, we have the vanishing
$\langle m_1,\cdots ,m_s\rangle=0$ of degree $d$ genus $0$ invariants
in $\mathbb{P}^n$, where $s=d+2$, with any $m_i$, and $d>1$.
\end{proposition}

\begin{proof}
We show that for any choice of $m_i$, the intersection theory problem
$\langle m_1,\cdots ,m_s\rangle$, interpreted as counting stable maps
``meeting'' generally chosen linear spaces of codimension $m_1$,
\dots, $m_s$,
 corresponds to the empty
intersection.
Let $c_i = n-m_i$ be the {\em dimension} of these linear spaces for
convenience;  $\sum c_i = n-2d+1$ from \eqref{eq:dec6}.

As $m_i \leq n$, we have
$$n (d+2) \geq \sum^s_{i=1} m_i = n + (n+1) d +
(d+2)-3 \quad \text{(using \eqref{eq:dec6}),}$$
from which $d \leq (n+1)/2 < n$.  The image of any degree $d$ stable map
lies inside a $\mathbb{P}^d$.  We show that there isn't even a $\mathbb{P}^d$
inside $\mathbb{P}^n$ meeting the (generally chosen)  linear spaces of
dimension $c_i$.
The codimension of the condition
(on
$\mathbb{G}(d,n)$)
that a $\mathbb{P}^d$ in $\mathbb{P}^n$ meet a $\mathbb{P}^{c_i}$ is $\max( 0,
n-d-c_i)$.
Then
\begin{eqnarray*}
\sum_{i=1}^s \max(0, n-d-c_i)
& \geq & \sum_{i=1}^s (n-d-c_i)  \\
& = & (n-d)(d+2) - (n-2d+1) \\
& >  & (d+1)(n-d)  \quad \text{(using $d>1$)} \\
& = & \dim \mathbb{G}(d,n)
\end{eqnarray*}
so there is no $\mathbb{P}^d$ in $\mathbb{P}^n$ meeting the desired linear
spaces,
and thus no degree $d$ stable map meeting these linear spaces.
\end{proof}

 \vskip 30pt

\appendix

\section{Combinatorial identities}

We prove Proposition \ref{comb}, using the Lagrange
inversion formula. The following form of Lagrange inversion formula
can be found in \cite[p. 38]{Sta}.
\begin{lemma}[Lagrange inversion formula]
\label{lag} Let $F(x)=a_1x+a_2x^2+\cdots\in \mathbb
C[[x]]$ be a power series with $a_1\ne0$ and $F^{-1}(x)\in \mathbb
C[[x]]$ be its inverse (defined by $F^{-1}(F(x))=x$). For
$k,n\in\mathbb Z$ we have
\begin{equation}\label{eqlag}
\frac{1}{n}[x^{n-k}]\left(\frac{x}{F(x)}\right)^n=\frac{1}{k}[x^n]F^{-1}(x)^k.
\end{equation}
\end{lemma}

We now prove Proposition \ref{comb}.

 Let $f(x)=1+\sum_{j=2}^\infty a_j x^j\in \mathbb
C[[x]]$ be as in Proposition \ref{comb}. Then $F(x)=x/f(x)$ is a
power series with $a_1\neq 0$, so we can apply Lemma \ref{lag}.
Taking $k=1$ in equation \eqref{eqlag}, we see that
$[x^2]F^{-1}(x)=\frac{1}{2} [x]f(x)^2=0$, so we have
\begin{align}\label{eqcomb2}
\frac{1}{F^{-1}(x)}&=\frac{1}{x+c_3x^3+c_4x^4+\cdots}\\
&=\frac{1}{x}-c_3x-c_4x^2-\cdots.\nonumber
\end{align}
Taking $k=1$ and $k=2$ in equation \eqref{eqlag} respectively, we
get
\begin{align*}
\frac{[x^{n+1}]f(x)^{n}}{n}&=-[x^n]\frac{1}{F^{-1}},\\
\frac{[x^{n+2}]f(x)^{n}}{n}&=-\frac{1}{2}[x^n]\frac{1}{F^{-1}}.
\end{align*}
Substituting the above two identities into equation \eqref{eqcomb} and
then applying equation \eqref{eqcomb2} to the summation term on the right
hand side, equation \eqref{eqcomb} becomes
\begin{align*}
-[x^{n+1}]\frac{1}{F^{-1}}&=\frac{1}{2}\sum_{j=1}^{n-1}
[x^{j}]\frac{1}{F^{-1}}[x^{n-j}]\frac{1}{F^{-1}}-\frac{1}{2}[x^{n}]\frac{1}{(F^{-1}(x)^2}\\
&=\left(\frac{1}{2}[x^{n}]\frac{1}{F^{-1}(x)^2}-[x^{n+1}]\frac{1}{F^{-1}}\right)
-\frac{1}{2}[x^{n}]\frac{1}{(F^{-1}(x)^2}\\
&=-[x^{n+1}]\frac{1}{F^{-1}}.
\end{align*}
So we have proved Proposition \ref{comb}.\qed
\medskip

We now present equivalent formulations of Proposition \ref{comb}
that may be useful elsewhere. We use the notation introduced
in \eqref{eqeu7}.

\begin{proposition} \label{comb2}
Let $\bold {a,b}\in \mathbb{N}^\infty$, $\bold c\in \mathbb{N}^\infty$ and $||\bold
c||\geq2$. Then the following identity holds

$$\frac{\left(|\bold c|+||\bold c||-3\right)!}{(|\bold c|-1)!}\cdot (||\bold c||-1)=
\frac{1}{2}\sum_{\substack{\bold c=\bold a+\bold b\\\bold a,\bold
b\neq0}}\binom{\bold c}{\bold a,\bold b}\frac{(|\bold a|+||\bold
a||-2)!\cdot (|\bold b|+||\bold b||-2)!}{(|\bold a|-1)!\cdot (|\bold
b|-1)!},$$
where $\binom{\bold c}{\bold a,\bold b}$ is defined as $\prod_{i \geq
  1} \frac { {\bold c}!}
{ {\bold a}! {\bold b}!
} = \prod_{i\geq1}\binom{
c_i}{a_i,b_i}$ (cf.\ \eqref{eqeu7}).
\end{proposition}
\begin{proof}
Take any $c=(c_1,c_2,\dots)\in \mathbb{N}^\infty$, compare the coefficient
$\prod_{j\geq2} a_j^{c_{j-1}}$ in both sides of equation
\eqref{eqcomb}. We have
$$\frac{\left(|\bold c|+||\bold c||-2\right)!}{(|\bold c|-1)!\bold c!}=
\frac{1}{2}\sum_{\substack{\bold c=\bold a+\bold b\\\bold a,\bold
b\neq0}}\frac{(|\bold a|+||\bold a||-2)!}{(|\bold a|-1)!\bold
a!}\frac{ (|\bold b|+||\bold b||-2)!}{(|\bold b|-1)!\bold
b!}+\frac{\left(|\bold c|+||\bold c||-3\right)!}{(|\bold c|-2)!\bold
c!}.$$ By moving the last term in the right hand side to the left, we get
the desired identity.
\end{proof}

A partition is a sequence of integers $\mu_1 \geq \mu_2 \geq \dots
\geq \mu_k > 0$. We write
$$
|\mu| = \mu_1+\cdots+\mu_k,\qquad\ell(\mu) = k.$$

Define $m_j(\mu)$ to be the number of $j$'s among $\mu_1, \dots,
\mu_k$,
$z_{\mu} = \prod_{j} m_j(\mu)!j^{m_j(\mu)}$, and $p_{\mu}=\prod_{j}p_j^{m_j(\mu)}$.

\begin{proposition}
$$\sum_{\ell(\mu)\geq2}\frac{(|\mu|+\ell(\mu)-3)!(\ell(\mu)-1)p_{\mu}}{(|\mu|-1)!z_{\mu}}
=\frac12\left(\sum_{\mu\neq
0}\frac{(|\mu|+\ell(\mu)-2)!p_{\mu}}{(|\mu|-1)!z_{\mu}}\right)^2$$
\end{proposition}
\begin{proof} Take $\bold c=(m_1(\mu),m_2(\mu),\dots)\in \mathbb{N}^{\infty}$. Then
the identity in the proposition is just a reformulation of Proposition \ref{comb2}.
\end{proof}

\vskip 30pt
\section{The differential polynomial $W_r(z)$}

From \S \ref{sectionpdo},  $\gamma_{-1}^{r+1}$ can
be expressed as a differential polynomial in
$\gamma_{-1}^1,\dots,\gamma_{-1}^{r-1}$. If  $2\leq i\leq r$, denote
by $p_i(r)$ the coefficient of $D^i\gamma_{-1}^{r+1-i}$ in the
resulting differential polynomial $S(\gamma_{-1}^{r+1})$. From the
proof of Proposition \ref{pdo6}, it is straightforward to obtain the
following recursive formula for $p_i(r)$,

\begin{align}\label{eqdiff1}
p_i(r)=&\frac{1}{r+1-i}[D^i
w_{-(r+1-i)}](\gamma_{-1}^{r+1}-\frac{r+1}{r}\gamma_{-2}^r-\frac{r+1}{2r}D\gamma_{-1}^r)\nonumber\\
&-\frac{1}{r+1-i}\sum_{j=2}^{i-1}\binom{r+1-j}{i+1-j}p_j(r)\\
=&\binom{r+1}{i}\frac{1}{r}\frac{i-1}{2(i+1)}-\frac{1}{r+1-i}\sum_{j=2}^{i-1}\binom{r+1-j}{i+1-j}p_j(r).\nonumber
\end{align}
We have proved $p_2(r)=\frac{r+1}{12}$ in Proposition \ref{pdo6}.
The relation of $p_i(r)$ to the coefficients of $W_r(z)$ is given
by
\begin{equation}\label{eqdiff5}
[z_{r-i}^{(i)}]W_r(z)=\frac{-\sqrt{-1}^i(r-i+1)}{r^{\frac{i}{2}-1}(r+1)}p_i(r).
\end{equation}

For $i \geq 2$, define quantities $C_i$ by
\begin{equation}\label{eqdiff2}
p_i(r)=\frac{1}{r}\binom{r+1}{i}C_i.
\end{equation}
We will see shortly that $C_i$ are in fact constants independent of
$r$.

Substituting \eqref{eqdiff2} into the recursion formula
\eqref{eqdiff1}, we get
\begin{align}
\frac{1}{r}\binom{r+1}{i}C_i&=\binom{r+1}{i}\frac{1}{r}\frac{i-1}{2(i+1)}\nonumber
-\frac{1}{r+1-i}\sum_{j=2}^{i-1}\binom{r+1-j}{i+1-j}\frac{1}{r}\binom{r+1}{j}C_j\\
C_i&=\frac{i-1}{2(i+1)}-\sum_{j=2}^{i-1}\binom{i+1}{j}\frac{C_j}{i+1}.\label{eqdiff3}
\end{align}
Hence by induction (starting from $C_0=-\frac{1}{2}$, $C_1=0$), we see
that the $C_i$ are constants. Using
the values of $C_0$ and $C_1$, we may simplify \eqref{eqdiff3} as
\begin{equation}\label{eqdiff4}
\sum_{j=0}^i\binom{i+1}{j}C_j=\frac{i-2}{2},\qquad i\geq 1.
\end{equation}

\begin{proposition}
Let $i \geq2$. Then $C_i=B_{i}$ the Bernoulli numbers. In particular,
$C_{2k+1}=0$.
\end{proposition}
\begin{proof}
We define a new sequence $C_j'$ by $C_0'=1, C_1'=-\frac12$ and
$C_j'=C_j,\,j\geq2$. From \eqref{eqdiff4}, we have
$$\sum_{j=0}^i\binom{i+1}{j}C_j'=0,\qquad i\geq 1,$$
which is the usual recursion for Bernoulli numbers.
Since $C_0'=B_0,C_1'=B_1$, we must have $C_j=C_j'=B_j$ for all
$j\geq2$.
\end{proof}

From \eqref{eqdiff5} and \eqref{eqdiff2}, we thus proved Proposition
\ref{pdo5}.

\vskip 30pt
\section{An identity of Bernoulli numbers}
Let $f_n(k),\,n\geq 1$ be given by the recursion
\begin{equation}\label{eq18}
f_{n+1}(k)=-(k+1)f_n(k+2)+(2k+1)f_n(k+1)-k f_n(k)
\end{equation}
starting with $f_1(k)=\frac{-1}{(k+1)(k+2)}$.

\begin{proposition}\label{bern} Let $n\geq 2$. We have
$
f_n(0)=-B_{n}/n$.
\end{proposition}

The rest of the appendix is devoted to proving the above proposition. First we record the following combinatorial identities.
\begin{gather}
\sum_{k=0}^n\binom{n}{k}(-1)^{n-k}\frac{1}{n+i-k}=\frac{n!(i-1)!}{(n+i)!},\label{eq23}\\
\sum_{i=w}^n\frac{(i-1)!}{(i-w)!}=\frac{n!}{(n-w)!w},\label{eq24}\\
\left(\frac{e^t-1}{t}\right)^w=\sum_{j=w}^\infty \frac{t^{j-w}}{j!}\sum_{s=0}^w (-1)^s\binom{w}{s}(w-s)^j.\label{eq25}
\end{gather}

Consider the generating function $F_n(x)=\sum_{k=0}^\infty f_n(k) x^k$, then we have
 $F_1(x)=\frac{-1}{x}+\frac{x-1}{x^2}\ln(1-x)$ and \eqref{eq18}
implies
\begin{align*}
F_n(x)&=-\frac{\partial}{\partial x}\left(\frac{F_{n-1}(x)}{x}\right)+2\frac{\partial}{\partial x}F_{n-1}(x)-\frac{1}{x}F_{n-1}(x)-x\frac{\partial}{\partial x}F_{n-1}(x)\\
&=\frac{-(x-1)^2}{x}F_{n-1}(x)+\frac{1-x}{x^2}F_{n-1}(x).
\end{align*}
More precisely, $F_n(x)$ from the above recursion differs from the true generating function by a finite sum of negative powers of $x$. Thus Proposition \ref{bern}
is equivalent to prove that the constant term of $-nF_n(x)$ equals $B_{n}$.

It is not difficult to see that $F_n(x)$ decomposes as
\begin{equation}
F_n(x)=G_n(x)+R_n(x)\ln(1-x),\quad n\geq1,
\end{equation}
where $G_n(x)$ and $R_n(x)$ are rational functions in $x$ satisfying the recursions
\begin{gather}
G_n(x)=\frac{-(x-1)^2}{x}\frac{\partial}{\partial x}G_{n-1}(x)+\frac{1-x}{x^2}G_{n-1}(x)+\frac{1-x}{x}R_{n-1}(x),\label{eq19}\\
R_n(x)=\frac{-(x-1)^2}{x}\frac{\partial}{\partial x}R_{n-1}(x)+\frac{1-x}{x^2}R_{n-1}(x).\label{eq20}
\end{gather}

We may solve \eqref{eq20} to get
\begin{equation}\label{eq21}
R_n(x)=(x-1)^n\sum_{i=1}^n (-1)^{i+1} a(n,i) x^{-n-i},
\end{equation}
where $a(n,i)$ is given by
$$a(n,i)=\sum_{w=0}^i (-1)^{i-w}\binom{n+i}{i-w}\sum_{s=0}^{w}(-1)^s\frac{(w-s)^{n+w}}{s!(w-s)!}.$$

From \eqref{eq19}, we may prove that
$[x^k]G_n(x)=0,\,\forall k\geq0$. By using \eqref{eq23}, \eqref{eq24}, \eqref{eq25}, the constant term of $-nF_n(x)$ equals
\begin{multline}\label{eq22}
[x^0](-n)R_n(x)\ln(1-x)=(-n)\sum_{i=1}^n(-1)^i a(n,i)\sum_{k=0}^n\binom{n}{k}(-1)^{n-k}\frac{1}{n+i-k}\\
=(-n)\sum_{i=1}^n(-1)^i a(n,i)\frac{n!(i-1)!}{(n+i)!}\\
=(-n)\sum_{w=1}^n\sum_{s=0}^w (-1)^s\binom{w}{s}\frac{(w-s)^{w+n}}{w!}(-1)^w\frac{n!}{(w+n)!}\sum_{i=w}^n\frac{(i-1)!}{(i-w)!}\\
=(-n)\sum_{w=1}^n\sum_{s=0}^w (-1)^s\binom{w}{s}\frac{(w-s)^{w+n}}{w!}(-1)^w\frac{n!}{(w+n)!}\frac{n!}{(n-w)!w}\\
=(-n)\sum_{w=1}^n(-1)^w\binom{n}{w}\frac{n!}{w}[t^n]\left(\frac{e^t-1}{t}\right)^w\\
=-\sum_{w=1}^n(-1)^w\binom{n}{w}n![t^{n-1}]\left(\left(\frac{e^t-1}{t}\right)^{w-1}\frac{d}{dt}\left(\frac{e^t-1}{t}\right)\right)\\
=-n![t^{n-1}]\left(\sum_{w=0}^n(-1)^w\binom{n}{w}\left(\frac{e^t-1}{t}\right)^{w}\frac{te^t-e^t+1}{t(e^t-1)}\right)
+n![t^{n-1}]\frac{te^t-e^t+1}{t(e^t-1)}\\
=-n![t^{n-1}]\left(\left(1-\frac{e^t-1}{t}\right)^{n}\frac{te^t-e^t+1}{t(e^t-1)}\right)
+n![t^{n}]\frac{te^t-e^t+1}{e^t-1}\\
=n![t^{n}]\frac{te^t-e^t+1}{e^t-1},
\end{multline}
where the last equation follows by noting
\begin{align*}
1-\frac{e^t-1}{t}&=\frac{-1}{2}t-\frac{1}{6}t^2+\cdots,\\
\frac{te^t-e^t+1}{t(e^t-1)}&=\frac{1}{2}+\frac{1}{12}t+\cdots.
\end{align*}

Finally, Proposition \ref{bern} follows from
\begin{equation}
1+\frac{te^t-e^t+1}{e^t-1}=\frac{t}{1-e^{-t}}=1+\frac{t}{2} +\sum_{n=2}^\infty \frac{B_n t^n}{n!}.
\end{equation}

\begin{remark}
Another way of proving Proposition \ref{bern} is by studying the function $h_j(k)=\prod_{i=1}^{2j} (k+i) f_j(k)$. Then \eqref{eq18} becomes
\begin{multline}\label{eq17}
h_{j+1}(k)=-(k+1)^2(k+2)h_j(k+2)\\+(k+1)(2k+1)(k+2j+2)h_j(k+1)-k(k+2j+1)(k+2j+2) h_j(k)
\end{multline}
starting with $h_1(k)=-1$, $h_2(k)=4k-2$, $h_3(k)=-36k^2+84k$.

We may prove from the recursion \eqref{eq17} (although more difficult) that $h_j(k)$ is a degree $j-1$ polynomial whose leading term equals $(-1)^j (j!)^2 k^{j-1}$
and the constant term equals $-(2j)!B_j/j$ when $j\geq 2$, as claimed in Proposition \ref{bern}.
\end{remark}

$$ \ \ \ $$

\end{document}